\documentclass[11pt]{article}

\usepackage{latexsym}
\usepackage{amssymb}
\usepackage{amsthm}
\usepackage{amscd}
\usepackage{amsmath}
\usepackage{mathrsfs}
\usepackage[all]{xy}
\usepackage{graphicx} 
\usepackage{hyperref} 
\usepackage[usenames,dvipsnames]{color}

\usepackage{color}

\theoremstyle{definition}
\newtheorem{theorem}{Theorem}[section]

\theoremstyle{definition}

\newtheorem{lemma}{Lemma}[section]
\theoremstyle{definition}

\theoremstyle{definition}
\newtheorem* {notation}{Notation}
\theoremstyle{definition}

\newtheorem{proposition}{Proposition}[section]
\newtheorem{corollary}{Corollary}[section]
\newtheorem* {remark}{Remark}
\newtheorem{example}{Example}[section]
\theoremstyle{definition}

\theoremstyle{definition}

\xyoption{dvips}

\usepackage{fullpage}

\numberwithin{equation}{section}

\newcommand{\Mat}{\mathrm{Mat}} 
\def\({\left(}
\def\){\right)}
  \newcommand{\sP}{\mathscr{S}} \newcommand{\FF}{\mathbb{F}}  \newcommand{\CC}{\mathbb{C}}     \newcommand{\cP}{\mathcal{P}} \newcommand{\cA}{\mathcal{A}}
\newcommand{\cQ}{\mathcal{Q}} \newcommand{\cK}{\mathcal{K}} \newcommand{\cO}{\mathcal{O}} \newcommand{\cC}{\mathcal{C}}
\newcommand{\cR}{\mathcal{R}}  \newcommand{\cS}{\mathcal{S}}\newcommand{\cT}{\mathcal{T}} 
 
\def\NN{\mathbb{N}}
    \def\ZZ{\mathbb{Z}}   \def\H{\mathcal{H}} \def\Ind{\mathrm{Ind}} \def\GL{\mathrm{GL}}       \def\spanning{\textnormal{-span}}   
\def\Irr{\mathrm{Irr}}  \def\wt{\widetilde}

   \newcommand{\fkn}{\mathfrak{n}}\newcommand{\supp}{\mathrm{supp}}

\newcommand{\One}{{1\hspace{-.14cm} 1}}

\newcommand{\SInd}{\mathrm{SInd}}

\newcommand{\fka}{\mathfrak{a}}
\newcommand{\fkh}{\mathfrak{h}}

\def\fk{\mathfrak}
\def\fkm{\mathfrak{m}}

\def\barr{\begin{array}}
\def\earr{\end{array}}
\def\ba{\begin{aligned}}
\def\ea{\end{aligned}}
\def\be{\begin{equation}}
\def\ee{\end{equation}}
\def\cS{\mathcal{S}}

\makeatletter
\renewcommand{\@makefnmark}{\mbox{\textsuperscript{}}}
\makeatother

\allowdisplaybreaks[1]

\UseCrayolaColors

\begin{document}
\title{A supercharacter analogue for normality}
\author{Eric Marberg\footnote{This research was conducted with government support under
the Department of Defense, Air Force Office of Scientific Research, National Defense Science
and Engineering Graduate (NDSEG) Fellowship, 32 CFR 168a.} \\ Department of Mathematics \\ Massachusetts Institute of Technology, United States \\ \tt{emarberg@math.mit.edu}}
\date{}

\maketitle

\begin{abstract}
Diaconis and Isaacs  define in \cite{DI06} a \emph{supercharacter theory} for algebra groups over a finite field by constructing certain unions of conjugacy classes called \emph{superclasses} and certain reducible characters called \emph{supercharacters}.  This work investigates the properties of algebra subgroups $H\subset G$ which are unions of some set of the superclasses of $G$; we call such subgroups \emph{supernormal}.  After giving a few useful equivalent formulations of this definition, we show that products of supernormal subgroups are supernormal and that all normal pattern subgroups are supernormal.  We then classify the set of supernormal subgroups of $U_n(q)$, the group of unipotent upper triangular matrices over the finite field $\FF_q$, and provide a formula for the number of such subgroups when $q$ is prime.   Following this, we give supercharacter analogues for Clifford's theorem and Mackey's ``method of little groups.''  Specifically, we show that a supercharacter restricted to a supernormal subgroup decomposes as a sum of supercharacters with the same degree and multiplicity.  We then describe how the supercharacters of an algebra group of the form $U_\fkn = U_\fkh \ltimes U_\fka$, where $U_\fka$ is supernormal and $\fka^2=0$, are parametrized by $U_\fkh$-orbits of the supercharacters of $U_\fka$ and the supercharacters of the stabilizer subgroups of these orbits.  

\end{abstract}

\section{Introduction}


Classifying the irreducible representations of $U_n(q)$, the group of $n\times n$ unipotent upper triangular matrices over a finite field $\FF_q$, is a well-known wild problem, provably intractable for arbitrary $n$.  Despite this, C. Andr\'e discovered a natural way of constructing certain sums of irreducible characters and certain unions of conjugacy classes of the group, which together form a useful approximation to its representation theory \cite{An95, An99, An01, An02}. In his 
PhD thesis \cite{Ya01}, N. Yan showed how to replace Andr\'e's construction with more elementary methods. This simplified theory proved to have both useful applications and a natural generalization.  For example, Arias-Castro, Diaconis, and Stanley 
\cite{ADS04} employed Yan's work in place of the usual irreducible character theory to study random walks on $U_n(q)$.

Later, Diaconis and Isaacs \cite{DI06} axiomatized the approximating approach to define the notion of a supercharacter theory for a finite group, in which supercharacters replace irreducible characters and superclasses replace conjugacy classes.  In addition, they generalized Andr\'e's original construction to define a supercharacter theory for algebra groups, a family of groups of the form $U_\fkn = \{1 + X : X \in \fkn \}$ where $\fkn$ is a nilpotent $\FF_q$-algebra.  The characters in this theory share many formal properties with the irreducible characters of a finite group.  For example, supercharacters are orthogonal with respect to the usual inner product on class functions and decompose the character of the regular representation.  Furthermore, restrictions and tensor products of supercharacters decompose as linear combinations of supercharacters with nonnegative integer coefficients, and there is a notion of superinduction that is adjoint to restriction on the space of superclass functions.  \cite{T09-2, MT09, T09, TV09} study these aspects of Diaconis and Isaacs' supercharacter theory in detail.

This work investigates a further analogy between supercharacters of algebra groups and irreducible characters of an arbitrary finite group, specifically with regard to normal subgroups.  The irreducible characters of a finite group are constant on conjugacy classes, and a subgroup which is a union of conjugacy classes is called \emph{normal}.  Mirroring this definition, we say that a subgroup of an algebra group is \emph{supernormal} if it is a union of superclasses.  A basic result in character theory states the we can define normal subgroups in another way: namely, as the subgroups of the form $\ker \chi_1 \cap \dots\cap \ker \chi_n$ for some set of irreducible characters $\chi_1,\dots,\chi_n$.  Thus, normal subgroups are the subgroups determined by the character table of a group.  We prove that supernormal subgroups  are likewise given by intersections of kernels of supercharacters, and thus are in the same way the subgroups determined by the supercharacter table of an algebra group.  

The goal of this paper is to investigate how this analogy between irreducible characters/normal subgroups and supercharacters/supernormal subgroups continues.  In particular, we study in depth how the restriction of a supercharacter to a supernormal subgroup parallels the restriction of an irreducible character to a normal subgroup.  Likewise, we present a supercharacter analogue for the classical result describing the irreducible characters of a semidirect product $G = H\ltimes A$ with an abelian, normal subgroup.  
 In the process of these investigations, we also study the properties of supernormal subgroups in their own right.  We show in particular that supernormal algebra subgroups correspond to two-sided ideals in the ambient nilpotent $\FF_q$-algebra, and we use this property to classify all supernormal subgroups of $U_n(q)$ when $q$ is prime.
 
 A more detailed outline of our results goes as follows.  Section \ref{prelim-sect} discusses the concept of a supercharacter theory of a finite group, and introduces the corresponding notion of a \emph{supernormal} subgroup.  The section continues with additional background material, including the definitions of algebra groups and pattern groups, and the particular supercharacter theory introduced by Diaconis and Isaacs \cite{DI06} for these groups.  We conclude these preliminaries by briefly discussing the restriction and superinduction of supercharacters of algebra groups.  

In Section \ref{supernormal1-sect} we investigate the supernormal subgroups of algebra groups.  We provide a useful characterization of such subgroups, showing that when $\fkn$ is an algebra over a field of prime order, the supernormal subgroups of $U_\fkn$ correspond to two-sided ideals in $\fkn$.  We also show that products of supernormal subgroups of algebra groups are supernormal, and define the lift of a supercharacter from a supernormal subgroup.  

Section \ref{supernormal2-sect} specializes this discussion to the case of pattern groups, a family of algebra groups defined by partial orderings.  We classify all normal pattern subgroups of a pattern group, and show that these subgroups are always supernormal.  In addition, we provide a combinatorial classification of all supernormal subgroups of $U_n(q)$, showing the number of such subgroups to be
\[ \sum_{0\leq i \leq j \leq k < n}  \frac{(-1)^{k-j}}{n} \binom{n}{k+1} \binom{n}{k}\binom{k}{j} \binom{j}{i}_q\] when $q$ is prime.  This formula also counts the number of two-sided ideals in the algebra of strictly upper triangular $n\times n$ matrices over $\FF_q$.

Section \ref{restrict-sect} discusses the restriction of a supercharacter to supernormal subgroup.  We prove a supercharacter version of Clifford's theorem on the restriction of an irreducible character to a normal subgroup, and also provide supercharacter analogues for results describing the restriction of an irreducible character to a normal subgroup of index two.  

Finally, in Section \ref{abelian-sect} we provide a supercharacter analogue for Mackey's ``method of little groups,'' which classifies the irreducible characters of a semidirect product of the form $G = H\ltimes A$ where $A$ is abelian and normal.  Specifically, we show how the supercharacters of an algebra group of the form $U_\fkn = U_\fkh \ltimes U_\fka$, where $U_\fka$ is supernormal and $\fka^2=0$, are parametrized by $U_\fkh$-orbits of the supercharacters of $U_\fka$ and the supercharacters of the stabilizer subgroups of these orbits.

\subsection*{Acknowledgements}

I  thank Nat Thiem for his helpful remarks and suggestions.

\section{Preliminaries}\label{prelim-sect}

This section presents the concept of a supercharacter theory and a supernormal subgroup, then defines algebra groups, pattern groups, and a specific supercharacter theory introduced by Diaconis and Isaacs in \cite{DI06}.  We also review the superclasses and supercharacters of the fundamental example $U_n(q)$, and discuss restriction and superinduction for algebra groups.

\subsection{Supercharacter Theories and Supernormal Subgroups} \label{abstract-sc}

Let $G$ be a finite group and write $\Irr(G)$ for the set of the group's (complex) irreducible characters.  A \emph{supercharacter theory} of $G$ is a set $\cS$ of characters of $G$ and a partition $\cS^\vee$ of the elements of $G$  satisfying the following conditions:
\begin{enumerate}
\item[(1)] $|\cS| = |\cS^\vee|$.  
\item[(2)] Each irreducible character of $G$ appears as a constituent of exactly one $\chi \in \cS$.
\item[(3)] Each $\chi \in \cS$ is constant on each set $\cK \in \cS$.
\item[(4)] The conjugacy class $\{1\} \in \cS^\vee$.
\end{enumerate} 
We call $\cS^\vee$ the set of \emph{superclasses} and $\cS$ the set of \emph{supercharacters} of the supercharacter theory $(\cS,\cS^\vee)$.  Each superclass is a union of conjugacy classes, and each supercharacter $\chi \in\cS$ is equal to a positive constant times  $\sum_{\psi  \in \Irr(G,\chi)} \psi(1) \psi$ where $\Irr(G,\chi)$ denotes the set of irreducible constituents of $\chi$ \cite[Lemma 2.1]{DI06}.  By condition (2), the sets $\Irr(G,\chi)$ for $\chi \in \cS$ form a partition of $\Irr(G)$, and consequently the supercharacters $\cS$ form an orthogonal basis for the space of \emph{superclass functions}, the complex valued functions on $G$ which are constant on the superclasses $\cS^\vee$.

Every finite group has two trivial supercharacter theories: the usual irreducible character theory and the supercharacter theory with $\cS = \{ \One, \rho_G - \One\}$ and $\cS^\vee = \{ \{1\}, G-\{1\}\}$, where $\rho_G$ denotes the character of the regular representation of $G$.    \cite{H09} discusses several methods of constructing additional supercharacter theories of an arbitrary finite group.

A subgroup of $G$ is normal if and only if it is the union of a set of conjugacy classes of $G$.  We have an equivalent characterization of normality in terms of the kernels of irreducible characters.  Recall that the kernel of a character $\chi$ of $G$ is the set $\ker \chi = \{ g \in G : \chi(g) = \chi(1)\}$.  This is just the kernel of any representation whose character is $\chi$, and so $\ker \chi$ is normal subgroup.  A subgroup of $G$ is normal if and only if it is the intersection of the kernels of some finite set of irreducible characters \cite[Proposition 17.5]{JL}; thus the normal subgroups of $G$ are the subgroups which we can construct from the character table of $G$.

We have a natural generalization of normality in an arbitrary supercharacter theory which preserves this property.  In particular, we say that a subgroup $H \subset G$ is \emph{supernormal} with respect to a supercharacter theory $(\cS,\cS^\vee)$ if $H$ is given by the union of a set of superclasses in $\cS^\vee$.  \cite{H09} calls such subgroups $(\cS,\cS^\vee)$-normal.  Then, as for the usual irreducible character theory of $G$, we have an alternate characterization in terms of the kernels of supercharacters.

\begin{proposition}\label{sc-prop}
Let $G$ be a finite group with a supercharacter theory $(\cS, \cS^\vee)$.  Then a subgroup $H\subset G$ is supernormal with respect to $(\cS, \cS^\vee)$ if and only if there exist supercharacters $\chi_1,\dots,\chi_n \in \cS$ with $H = \bigcap_{i=1}^n \ker \chi_i$.
\end{proposition}

Thus the supernormal subgroups of $G$ with respect to an arbitrary supercharacter theory are those subgroups which can can construct from the supercharacter table of $G$$-$i.e., the table whose rows are indexed by $\cS$ and whose columns are indexed by $\cS^\vee$, and whose entries record the value of a supercharacter $\chi$ at a superclass $\cK$.

\begin{proof}[Proof of Proposition \ref{sc-prop}]
The map $G\rightarrow \CC^n$ given by $g\mapsto (\chi_1(g),\dots, \chi_n(g))$ is constant on superclasses, so its kernel, which is precisely the intersection $\bigcap_{i=1}^n \ker \chi_i$, is a union of superclasses and therefore a supernormal subgroup.  Conversely, suppose $H$ is an arbitrary supernormal subgroup.  Then $H$ is normal, so we can consider the quotient group $G/H$.  Given a character $\psi$ of $G/H$, let $\wt \psi$ denote its lift to $G$: this is the character of $G$ defined by $\wt \psi(g) = \psi(gH)$.  If $\psi$ is irreducible then $\wt \psi$ is irreducible, and $\ker \wt \psi \supset H$.  Let $\rho_{G/H}$ denote the character of the regular representation of $G/H$, and observe that since $H$ is supernormal, $\wt \rho_{G/H}$ is constant on the superclasses $\cS^\vee$.  Therefore for some constants $c_\chi \in \CC$ we have 
\[ \sum_{\chi \in \cS} c_\chi \chi = \wt \rho_{G/H} = \sum_{\psi \in \Irr(G/H)} \psi(1) \wt \psi.\]  Since the constituents of distinct supercharacters are disjoint, it follows from this equation that if $\psi \in \Irr(G/H)$ and $\chi \in \cS$ has $\wt \psi$ as a constituent, then every constituent of $\chi$ is a lift of an irreducible character of $G/H$, and so $\ker \chi \supset H$.  Enumerate the irreducible characters of $G/H$ as $\psi_1,\dots,\psi_s$ and for each $i$ let $\chi_i \in \cS$ be the unique supercharacter with $\wt \psi_i$ as a constituent.  Note that $\bigcap_{i=1}^s \ker \psi_i = \{H\} \subset G/H$ so $\bigcap_{i=1}^s \ker \wt \psi_i = H \subset G$.   
Now, since the kernel of $\chi_i$ is the intersection of the kernels of its constituents, we have $\bigcap_{i=1}^s \ker \chi_i \subset \bigcap_{i=1}^s \ker \wt \psi_i = H$.  On the other hand, $\ker \chi_i \supset H$ for all $i$, so $\bigcap_{i=1}^s \ker \chi_i = H$. 
\end{proof}

In this work, we study the supernormal subgroups of a particular supercharacter theory introduced by Diaconis and Isaacs \cite{DI06}.  Before introducing this supercharacter theory, we 
must define \emph{algebra groups,} the family of groups to which the theory applies. This is the goal of the next 
section.

\subsection{Algebra Groups}

Fix a finite field $\FF_q$ with $q$ elements and let $\fkn$ denote a nilpotent $\FF_q$-algebra.  In this work, all algebras are finite dimensional and associative, and are defined over a fixed ambient finite field $\FF_q$.  The \emph{algebra group} $U_\fkn$ corresponding to $\fkn$ is the group of formal sums $U_\fkn = \{ 1 + X : X \in \fkn \}$ with multiplication defined by 
\[ (1+X)(1+Y) = 1 + X + Y + XY,\qquad\text{for }X,Y\in \fkn.\] 
%




The group $U_\fkn$ acts on $\fkn$ on the left and right by the formal multiplications
\[\label{fkn-action} (1+X)Y = Y +XY\qquad\text{and}\qquad Y(1+X) = Y + YX,\qquad\text{for }X,Y \in \fkn.\]  Let $\fkn^*$ denote the dual space of $\fkn$; i.e., the set of $\FF_q$-linear maps $\fkn \rightarrow \FF_q$.  Then we have analogous left and right actions of $U_\fkn$ on $\fkn^*$ given by defining $g\lambda$ and $\lambda g$ for $g \in U_\fkn$ and $\lambda \in \fkn^*$ to be the functionals with
\[\label{star-action} g\lambda(X) = \lambda(g^{-1}X)\qquad\text{and}\qquad \lambda g(X) = \lambda(Xg^{-1}),\qquad\text{for }X \in \fkn.\] Both of these actions \emph{commute} (or are said to be \emph{compatible}) in the sense that $(gX)h = g(Xh)$ and $(g\lambda)h = g(\lambda h)$ for $g,h \in U_\fkn$, $X \in \fkn$, and $\lambda \in \fkn^*$.  Hence in both cases we may remove all parentheses without introducing ambiguity.  

Given $X \in \fkn$ and $\lambda \in \fkn^*$, we denote the corresponding left, right, and two-sided $U_\fkn$-orbits by $U_\fkn X$, $XU_\fkn$, $U_\fkn X U_\fkn$ and $U_\fkn\lambda$, $\lambda U_\fkn$, $U_\fkn\lambda U_\fkn$.   These orbits have the following useful properties:
\begin{enumerate}
\item[(1)] $|U_\fkn X U_\fkn| = \displaystyle\frac{|U_\fkn X| | X U_\fkn|}{|U_\fkn X \cap X U_\fkn|}$ and $|U_\fkn \lambda U_\fkn| = \displaystyle \frac{|U_\fkn \lambda| | \lambda U_\fkn|}{|U_\fkn \lambda \cap \lambda U_\fkn|}$.
\item[(2)] $|U_\fkn \lambda| = |\lambda U_\fkn|$.
\item[(3)] The numbers of left, right, and two-sided $U_\fkn$-orbits in $\fkn$ and $\fkn^*$ are respectively equal. 
\item[(4)] The sets $(U_\fkn X-X)$ and $(XU_\fkn -X)$ are subspaces of $\fkn$.  Likewise, the sets $(U_\fkn \lambda - \lambda)$ and $(\lambda U_\fkn -\lambda)$ are subspaces of $\fkn^*$. 
\end{enumerate}
These results derive from Lemmas 3.1, 4.1, and 4.2 in \cite{DI06}.

\def\cov{\mathrm{cov}}

\subsection{Pattern Groups}

A particularly tangible class of algebra groups, known as \emph{pattern groups}, can be defined in terms of partial orderings.  Fix a positive integer $n$ and let $[n] = \{1,2,\dots,n\}$.  We denote by $[[n]]$ the set of positions above the diagonal in an $n\times n$ matrix:
\[ [[n]] = \{ (i,j) : 1\leq i < j \leq n\} .\]  In this work, by a \emph{poset} $\cP$ on $[n]$ shall we mean a subset $\cP \subset[[n]]$ such that if $(i,j), (j,k) \in \cP$ then $(i,k) \in \cP$.  The requirement $\cP \subset [[n]]$ is somewhat nonstandard; we include it to ensure that our pattern groups are subgroups of $U_n(q)$.

A poset  $\cP$ corresponds to the strict partial ordering $\prec$ of the set $\{1,2,\dots,n\}$ defined by setting $i \prec j$ if and only if $(i,j) \in \cP$.  
 We visually depict $\cP$ via its Hasse diagram: the directed graph whose vertices are $1,2,\dots,n$ and whose directed edges are the ordered pairs $(i,k)\in \cP$ for which no $j$ exists with $(i,j),(j,k) \in \cP$.  For example, we can define the poset $\cP = \{ (1,3), (1,4), (2,3), (2,4), (3,4) \}$ on $[4]$ by writing
\[ \cP\ =\  \xy<0.25cm,0.8cm> \xymatrix@R=.3cm@C=.3cm{
  &  4  \\
  &  3  \ar @{-} [u]   \\
1 \ar @{-} [ur]   & 2 \ar @{-} [u] 
}\endxy \]

Fix a finite field $\FF_q$.  Given a poset $\cP$ on $[n]$, define $\fkn_\cP$ as the $\FF_q$-vector space of strictly upper triangular $n\times n$ matrices 
\[ \fkn_\cP = \{ X \in \Mat_n(\FF_q) : X_{ij} = 0\text{ if }(i,j) \notin \cP \}.\]  
As usual, we let $\fkn_\cP^*$ denote the dual space of $\FF_q$-linear functionals on $\fkn_\cP$.  

\begin{notation}
Given a matrix $X \in \fkn_\cP$ and a functional $\lambda \in \fkn_\cP^*$, define 
\[ \ba \supp(X) &= \{ (i,j)  \in \cP: X_{ij} \neq 0\}, \\
\supp(\lambda) &= \{(i,j) \in \cP : \lambda_{ij} \neq 0\},\ea\] where $\lambda_{ij} \overset{\mathrm{def}} = \lambda(e_{ij})$ and $e_{ij} \in \fkn_\cP$ is the elementary matrix with 1 in position $(i,j)$ and 0 in all other positions.
 \end{notation}

The \emph{pattern group} $U_\cP$ is the algebra group $U_\cP = U_{\fkn_\cP}$ over $\FF_q$; i.e., the group of unipotent upper triangular matrices 
\[ U_\cP = \{ 1 + X  : X \in \fkn_\cP \}.\]  
Note that under our definitions, the set $\cP$ can serve as a poset on $[n]$ for any sufficiently large integer $n$.  Thus, implicit in the notations $\fkn_\cP$, $\fkn_\cP^*$, $U_\cP$ is the choice of a dimension $n$ corresponding to $\cP$.  This choice has no effect on the isomorphism class of $U_\cP$, however.

\begin{notation} When $\cP = [[n]]$, we write $U_n(q)$, $\fkn_n(q)$, $\fkn_n^*(q)$ instead of $ U_{[[n]]}$, $\fkn_{[[n]]}$, $\fkn_{[[n]]}^*$.
\end{notation}

\subsection{Superclasses and Supercharacters of Algebra Groups}\label{sct-ag}

Fix a nilpotent $\FF_q$-algebra $\fkn$ and consider the algebra group $U_\fkn$.  In this section we define a set of superclasses and supercharacters of $U_\fkn$ which form a supercharacter theory in the sense of Section \ref{abstract-sc}.  Diaconis and Isaacs \cite{DI06} first defined this particular supercharacter theory as a generalization of the work of Andr\'e \cite{ An95} and Yan \cite{Ya01}.  


 The map $X\mapsto 1 +X$ gives a bijection $\fkn \rightarrow U_\fkn$, and we define the \emph{superclasses} of $U_\fkn$ to be the sets formed by applying this map to the two-sided $U_\fkn$-orbits in $\fkn$.  The superclass of $U_\fkn$ containing $g \in U_\fkn$, which we denote $\cK_\fkn^g$, is thus  the set
\[ \label{superclass-defin} \cK_\fkn^g \overset{\mathrm{def}}= \{ 1+x(g-1)y : x,y \in U_\fkn\}.\]  Each superclass is a union of conjugacy classes, and one superclass consists of just the identity element of $U_\fkn$.  

Fix a nontrivial group homomorphism $\theta : \FF_q^+\rightarrow \CC^\times$.  The \emph{supercharacters} of $U_\fkn$ are then the functions $\chi^\lambda_\fkn : U_\fkn \rightarrow \CC$ indexed by $\lambda \in \fkn^*$, defined by the formula  
\be\label{formula}\chi^\lambda_\fkn(g) = \frac{|U_\fkn \lambda |}{|U_\fkn \lambda U_\fkn|} \sum_{\mu \in U_\fkn\lambda U_\fkn} \theta\circ \mu(g-1),\qquad\text{for }g \in U_\fkn.\ee 
It follows from this definition that $\chi_\fkn^\lambda = \chi_\fkn^\mu$ if and only if $\mu \in U_\fkn \lambda U_\fkn$, and that supercharacters are constant on superclasses.  The function $\chi_\fkn^\lambda$ is the character of the left $U_\fkn$-module 
\[ V^\lambda_\fkn = \CC\spanning\{ v_\mu : \mu \in U_\fkn \lambda \},\qquad\text{where } g v_\mu  = \theta\circ \mu\(1-g^{-1}\) v_{g\mu}\text{ for }g\in U_\fkn.\] This module can be realized as an explicit submodule of $\CC U_\fkn$ by setting 
\[v_\mu =  \sum_{g \in U_\fkn} \theta\circ \mu(1-g) g \in \CC U_\fkn.\]  
For $\lambda, \mu \in \fkn^*$, 
\[ \langle \chi^\lambda_\fkn, \chi^\mu_\fkn \rangle_{U_\fkn} = \left\{ \begin{array}{ll} |U_\fkn\lambda \cap \lambda U_\fkn|,&\text{if }\mu \in U_\fkn\lambda U_\fkn, \\ 0,&\text{otherwise,}\end{array}\right.\quad\text{where }\langle \chi, \psi \rangle_{U_\fkn} = \frac{1}{|U_\fkn|} \sum_{g\in U_\fkn} \chi(g)\overline{\psi(g)}.\] Thus $\chi_\fkn^\lambda$ is irreducible if and only if $U_\fkn \lambda \cap \lambda U_\fkn  = \{\lambda\}$, and distinct supercharacters are orthogonal.

The numbers of superclasses and supercharacters are equal to the numbers of two-sided $U_\fkn$ orbits in $\fkn$ and $\fkn^*$, and hence are the same.  Furthermore, the character $\rho_{U_\fkn}$ of the regular representation of $U_\fkn$ decomposes as 
\[ \rho_{U_\fkn} = \sum_\lambda \frac{|U_\fkn \lambda U_\fkn|}{|U_\fkn \lambda|} \chi_\fkn^\lambda\] where the sum is over a set of representatives $\lambda$ of the two-sided $U_\fkn$ orbits in $\fkn^*$, and so each irreducible character of $U_\fkn$ appears as a constituent of a unique supercharacter.  We conclude that the supercharacters and superclasses defined in this way form a supercharacter theory of $U_\fkn$.

\begin{notation} 
When dealing with pattern groups, we typically replace the subscript $\fkn$, indicating the ambient nilpotent $\FF_q$-algebra, with $\cP$, indicating the ambient poset.  So $\cK^g_\fkn$ and $\chi^\lambda_\fkn$ become $\cK^g_\cP$ and $\chi^\lambda_\cP$.  When the nilpotent $\FF_q$-algebra $\fkn$ or poset $\cP$ is clear from the context, we may in turn abbreviate these various symbols as just $\cK^g$ and $\chi^\lambda$.
\end{notation}

\subsection{Superclasses and Supercharacters of $U_n(q)$}\label{sc-U_n-sect}

The supercharacter theory described in the preceding section arose as a generalization of a specific
attempt to approximate the irreducible characters of $U_n(q)$. The classification of this group's conjugacy 
classes and irreducible representations is a wild problem, but the classification of its 
superclasses and supercharacters has a highly satisfactory combinatorial answer, which provides  a fundamental example we will often consult.

Recall, $U_n(q)$ denotes the group of  upper triangular $n\times n$ matrices over $\FF_q$ with ones on the diagonal, $\fkn_n(q)$ denotes the nilpotent algebra of strictly upper triangular $n\times n$ matrices over $\FF_q$, and $\fkn_n^*(q)$ denotes the dual space of $\fkn_n(q)$.  

\begin{notation} Throughout, we let $e_{ij} \in \fkn_n(q)$ denote the $n\times n$ matrix with 1 in position $(i,j)$ and 0 is all other positions, and we let $e_{ij}^* \in \fkn_n^*(q)$ denote the linear functional defined by $e_{ij}^*(X) = X_{ij}$ for $X \in \fkn_n(q)$.
\end{notation}

Given a positive integer $n$, define 
 \[ \ba \sP_n(q) &= \{ \lambda \in \fkn_n(q) : \supp(\lambda)\text{ contains at most one position in each row and column} \}, \\
 \sP_n^*(q) &= \{ \lambda \in \fkn_n^*(q) : \supp(\lambda)\text{ contains at most one position in each row and column} \}.\ea\] 
We can identify elements in these sets with \emph{$\FF_q$-labeled set partitions of $[n]$}. A set partition $\lambda =\{\lambda_1,\dots,\lambda_\ell\}$ of $[n]$ is just a set of nonempty disjoint sets $\lambda_i$ whose union is $[n] = \{1,2,\dots,n\}$.   We say that $\lambda$ is $\FF_q$-labeled if we have a map assigning to each pair of consecutive integers in the parts $\lambda_i$ an element of $\FF_q^\times$.  This definition gives a $q$-analogue for set partitions; in particular, when $q=2$ set partitions and $\FF_q$-labeled set partitions are really the same objects.  
We mention there are multiple $q$-analogues for set partitions in the literature; see for example \cite{q-analog1,q-analog2}

An element $\lambda$ of $\sP_n(q)$ or $\sP_n^*(q)$ corresponds to the set partition of $[n]$ whose parts are the equivalence classes in $[n]$ under the relation $\sim$ defined by setting $i \sim j$ if $(i,j ) \in \supp(\lambda)$ or $(j,i) \in \supp(\lambda)$ and extending transitively.  This set partition comes with the natural $\FF_q$-labeling given by assigning each pair $(i,j)$ the value $\lambda_{ij} \in \FF_q^\times$.

Yan showed in \cite{Ya01} that the superclasses and supercharacters of $ U_n(q)$ are indexed $\sP_n(q)$ and $\sP_n^*(q)$; explicitly, we have bijections
\be\label{U_n(q)-classification} \barr{ccc}\sP_n(q) & \to &  \left\{ \barr{c} \text{Superclasses} \\ \text{of $ U_n(q)$} \earr\right\} \\ 
\lambda & \mapsto & \cK^{1+\lambda} \earr
\qquad\text{and}\qquad
\barr{ccc}\sP_{n}^*(q) & \to &  \left\{ \barr{c} \text{Supercharacters} \\ \text{of $ U_n(q)$} \earr\right\} \\ 
\lambda & \mapsto & \chi^\lambda. \earr
\ee  Andr\'e proved the character result earlier in \cite{An95}.  It follows that the number of superclasses and supercharacters of $ U_n(q)$ is given by $|\sP_n(q)| = |\sP_n^*(q)|$, which we denote by $B_n(q)$.  For $q=2$, $B_n(q)$ is just the $n$th Bell number, and for arbitrary $q$, Yan \cite{Ya01} showed  that $B_n(q)$ satisfies the recurrence 
\be\label{B_q-def} \ba B_0(q) &= 1, \\ 
B_{n+1}(q) &= \sum_{k=0}^{n} \binom{n}{k} (q-1)^k B_{n-k}(q),&&\qquad\text{for }n\geq 0.
\ea\ee 
The value of a supercharacter indexed by $\lambda \in \sP_n^*$ at a superclass indexed by $\mu \in \sP_n$ has the  formula
\be\label{added-formula} \chi^\lambda(1 + \mu) = \left\{\barr{ll} \displaystyle \(\prod_{(i,l) \in \supp(\lambda)} q^{l-i-1-f_\mu(i,l)} \) \theta \circ\lambda(\mu) ,&\barr{l} \text{if }(i,j),(j,k) \notin\supp(\mu)\text{ whenever } \\ \text{$i<j<k$ and $(i,k) \in \supp(\lambda)$}\earr \\ \\ 0,&\text{otherwise}\earr\right.\ee
where $f_\mu(i,l) \overset{\mathrm{def}}= |\{ (j,k) \in \supp(\mu) : i<j<k<l\}|$.
Andr\'e \cite{An95} first derived this remarkable closed form formula, but with some restrictions on the characteristic of $\FF_q$. Yan \cite{Ya01} later removed these restrictions, proving that the formula holds over all finite fields.

\subsection{Restriction and Superinduction}\label{resind-sect}

Before proceeding, we briefly discuss how one can restrict and induce a supercharacter to and from an algebra subgroup.

\begin{notation} 
Given sets $S'\subset S$, we denote the restriction of a function $f$ on $S$ to $S'$ by $f \downarrow S'$.  
\end{notation}

If  $\fkm \subset \fkn$ are nilpotent $\FF_q$-algebras and $\chi:U_\fkn \rightarrow \CC$ is a superclass function, then the restriction $\chi \downarrow U_\fkm$ is a superclass function of $U_\fkm$, and so is equal to a linear combination of supercharacters of $U_\fkm$.  If $\chi$ is a supercharacter of $U_\fkn$, then the restriction $\chi \downarrow U_\fkm$ is  a linear combination of supercharacters $U_\fkm$ with nonnegative integer coefficients \cite[Theorem 6.4]{DI06}.

The usual definition of induction does not in general send a supercharacter to a $\ZZ_{\geq 0}$-linear combination of supercharacters, or even to a superclass function.  To remedy this, Diaconis and Isaacs define \emph{superinduction}  in \cite{DI06} as a map $\SInd$ adjoint to restriction which takes superclass functions of an algebra subgroup $U_\fkm \subset U_\fkn$ to superclass functions of $U_\fkn$.  Explicitly, if $U_\fkm\subset U_\fkn$ are algebra groups and $\chi : U_\fkm \rightarrow \CC$ is a superclass function, then we define
\[\label{sind} \SInd_{U_\fkm}^{U_\fkn} (\chi)(g) = \frac{1}{|U_\fkm||U_\fkn|} \sum_{x,y \in U_\fkn} \overset{\circ}\chi(x(g-1)y+1),\quad\text{where}\quad \overset\circ \chi(z) = \left\{\barr{ll} \chi(z), & z \in U_\fkm \\ 0,&\text{otherwise}\earr\right.\] for $g\in U_\fkn$.  We see from this formula that the degree of the superinduced function is $\SInd_{U_\fkm}^{U_\fkn} (\chi)(1) = \frac{|U_\fkn|}{|U_\fkm|} \chi(1)$.  Since each value of $\SInd_{U_\fkm}^{U_\fkn} (\chi)$ is given by averaging $\chi$ over a superclass of $U_\fkn$, superinduction takes superclass functions of $U_\fkm$ to superclass functions of $U_\fkn$.  At the same time, superinduction is adjoint to restriction on the space of superclass functions, in the sense that 
\be\label{adjoint}\left \langle \SInd_{U_\fkm}^{U_\fkn} (\chi), \psi \right\rangle_{U_\fkn} = \left\langle \chi, \psi \downarrow U_\fkm \right\rangle_{U_\fkm}\ee for all superclass functions $\chi : U_\fkm \rightarrow \CC$ and $\psi : U_\fkn \rightarrow \CC$.  

While the restriction of a supercharacter to a subgroup is always a character, the same is not true of $\SInd_{U_\fkm}^{U_\fkn} (\chi)$, even if $\chi$ is a supercharacter of $U_{\fkm}$.  It follows from the reciprocity identity (\ref{adjoint}), nevertheless, that if $\chi$ is a supercharacter of $U_\fkm$ then $\SInd_{U_\fkm}^{U_\fkn} (\chi)$ is a linear combination of supercharacters of $U_\fkn$ with positive rational coefficients given by (possibly negative) powers of $q$.  The following lemma says something a bit more descriptive about how a superinduced supercharacter decomposes.  This result will be of some use later.

\begin{lemma}\label{sind-lemma}
Let $\fkm \subset \fkn$ be nilpotent $\FF_q$-algebras with $\mu \in \fkm^*$.  Then 
\[ \SInd_{U_\fkm}^{U_\fkn}(\chi_\fkm^\mu) = \sum_{\substack{\lambda \in \fkn^*\\ \lambda \downarrow \fkm  \in U_\fkm \mu U_\fkm}} \frac{|U_\fkm \mu|}{|U_\fkm \mu U_\fkm| |U_\fkn \lambda|}  \chi_\fkn^\lambda .\]
 \end{lemma}
 

%


\begin{proof}
Let $\fkm^\perp = \{ \gamma \in \fkn^* : \ker\gamma \supset \fkm\}$, and observe that $|\fkm^\perp| = \frac{|\fkn|}{|\fkm|} = \frac{|U_\fkn|}{|U_\fkm|}$.  For any $\nu \in \fkm^*$,  we have $\{ \lambda \in \fkn^*: \lambda \downarrow \fkm = \nu\} = \wt \nu + \fkm^\perp$ where $\wt \nu \in \fkn^*$ is an arbitrary functional with $\wt \nu \downarrow \fkm = \nu$.  Thus if $\nu \in \fkm^*$ and $X \in \fkn$,  then
\[ \frac{|U_\fkm|}{|U_\fkn|}\sum_{\substack{\lambda \in \fkn^*\\ \lambda \downarrow \fkm = \nu}} \theta \circ \lambda(X)  = \left\{\barr{ll}  \theta\circ \nu(X), & \text{if $X \in \fkm$,} \\ 0,&\text{otherwise,}\earr\right.
\quad \text{since} 
\quad\sum_{\substack{\lambda \in \fkn^*\\ \lambda \downarrow \fkm=\nu}} \theta \circ \lambda(X) =  \theta\circ \wt \nu(X) \sum_{\eta \in \fkm^\perp} \theta\circ \eta(X)\] and  
 $\sum_{\eta \in \fkm^\perp} \theta\circ \eta(X) =0$ if $X \notin\fkm$ by by standard character orthogonality relations.  It now follows from (\ref{formula}) that if $\mu \in \fkm^*$ and $g \in U_\fkn$, then
\[ \ba\SInd_{U_\fkm}^{U_\fkn}(\chi_\fkm^\mu)(g) 
&=  \frac{1}{ |U_\fkm| |U_\fkn|} \sum_{x,y \in U_\fkn}  \frac{|U_\fkm\mu|}{|U_\fkm \mu U_\fkm|} \sum_{\nu\in U_\fkm\mu U_\fkm}  \frac{|U_\fkm|}{|U_\fkn|}\sum_{\substack{\lambda \in \fkn^*\\ \lambda \downarrow \fkm = \nu}} \theta \circ \lambda(x(g-1)y) 
\\
 &=  \frac{1}{|U_\fkn|^2} \frac{|U_\fkm\mu|}{|U_\fkm \mu U_\fkm|}  \sum_{x,y \in U_\fkn}    \sum_{\substack{\lambda \in \fkn^*\\ \lambda \downarrow \fkm  \in U_\fkm \mu U_\fkm}} \theta \circ \lambda(x(g-1)y)
\\
 &= \frac{|U_\fkm\mu|}{|U_\fkm \mu U_\fkm|} \sum_{\substack{\lambda \in \fkn^*\\ \lambda \downarrow \fkm \in U_\fkm \mu U_\fkm}} \frac{1}{|U_\fkn \lambda U_\fkn|} \sum_{\gamma \in U_\fkn \lambda U_\fkn}  \theta \circ \gamma(g-1) 
 =   
  \sum_{\substack{\lambda \in \fkn^*\\ \lambda \downarrow \fkm  \in U_\fkm \mu U_\fkm}} \frac{|U_\fkm \mu|}{|U_\fkm \mu U_\fkm| |U_\fkn \lambda|} \chi_\fkn^\lambda(g)
.\ea\] 
\end{proof}

Before moving on, we describe one additional property of superinduction, which gives a supercharacter analogue for Mackey's theorem on the restriction of induced characters.  This classical result goes as follows.  Let $H,K \subset G$ be groups, and choose a character $\chi$ of $H$ and an element $s \in G$.  Define $H^s = s^{-1} H s$ and $D_s = H^s \cap K$, and let $\chi_s = \chi^s \downarrow D_s$, where $\chi^s$ is the character of $H^s$ defined by 
\[ \chi^s(x) = \chi( sxs^{-1}),\qquad\text{for }x \in H^s.\]  If $G = \bigcup_{s \in I} HsK$ is a decomposition of $G$ into double cosets, then Mackey's theorem states that
\[ \Ind_{H}^{G}(\chi) \downarrow K = \sum_{s \in I} \Ind_{D_s}^K (\chi_s).\]  Retaining this notation, we have a similar result for supercharacters and superinduction.

\begin{proposition}
 Let $H,K\subset G$ be algebra groups and let $G = \bigcup_{s \in I} HsK$ be a decomposition of $G$ into double cosets.  If $\chi$ is a supercharacter of $H$, then 
 \[ \SInd_{H}^{G}(\chi) \downarrow K = \sum_{s \in I} \frac{|HsK|}{|G|} \SInd_{D_s}^K (\chi_s).\]
\end{proposition}

\begin{proof}
For each $s \in I$, let $R_s$ be a set of right coset representatives of $D_s$ in $K$, so that $K = \bigcup_{t \in R_s} D_s t$.  Then $HsK = \bigcup_{t \in R_s} Hst$ is a partition into disjoint sets, and so for $k \in K$,
\[\ba \SInd_H^G(\chi)(k) &=\frac{1}{|H||G|} \sum_{\substack{x,y \in G \\ x(k-1)y+1 \in H}} \chi \(x(k-1)y+1\) = \frac{|H|}{|G|} \sum_{s \in I} \sum_{\substack{t,u \in R_s \\ t(k-1)u^{-1}+1 \in H^s }} \chi^s\(t(k-1)u^{-1}+1\)  \\
& 
= \sum_{s \in I}  \frac{|H|}{|G||D_s|^2} \sum_{\substack{ x,y \in K \\ x(k-1)y+1 \in D_s}} (\chi^s\downarrow D_s)\(x(k-1)y+1\)
= \sum_{s \in I}  \frac{|H||K|}{|G||D_s|} \SInd_{D_s}^K( \chi_s) (k).
\ea
\]  The proposition now follows by noting that $|HsK| = |H^s K|  = \frac{|H^s||K|}{|H^s \cap K|}= \frac{|H||K|}{|D_s|}$.
\end{proof}

\section{Supernormal Algebra Subgroups}\label{supernormal1-sect}

We recall from Section \ref{abstract-sc} that if $G$ is a finite group with a supercharacter theory $(\cS,\cS^\vee)$, then a subgroup $H$ is \emph{supernormal} if it is given by the union of a set of superclasses in $\cS^\vee$.  In this section we investigate this definition in the context of algebra groups and the supercharacter theory introduced in Section \ref{sct-ag}.  To begin, we observe that 
a supernormal algebra subgroup is automatically normal, but the converse is not true in general.

\begin{example}\label{converse-fails-example}
 Let $\fkm\subset\fkn$ be the nilpotent  $\FF_q$-algebras 
\[ \fkm = \left\{ \(\barr{cccc} 1 & a & b & c \\ 0 & 1 & 0 & -b \\ 0 & 0 & 1 & a \\ 0 & 0 & 0 & 1 \earr\) : a,b,c \in \FF_q \right\} \quad\text{and}\quad 
\fkn  = \fkn_4(q) = \left\{ \(\barr{cccc} 1 & a & b & c \\ 0 & 1 & d & e \\ 0 & 0 & 1 & f \\ 0 & 0 & 0 & 1 \earr\) : a,b,c,d,e,f \in \FF_q \right\}.
\] One can check that $U_\fkm \vartriangleleft U_\fkn$; in particular, it suffices to confirm that $(1+te_{i,i+1})X(1-te_{i,i+1}) \in \fkm$ for an arbitrary $X \in \fkm$, $t\in \FF^\times$, and $i=1,2,3$.  However, $U_\fkm$ is not supernormal in $U_\fkn$ since $(1+e_{23})X \notin \fkm$ for $X \in \fkm$ with $X_{12} = X_{34}\neq 0$.  
\end{example}

If  $\fkm \subset \fkn$ are  nilpotent  $\FF_q$-algebras, then by definition the algebra subgroup $U_\fkm \subset U_\fkn$ is {supernormal} if and only if $gXh \in \fkm$ for all $g,h \in U_\fkn$ and $X \in \fkm$.   We expand on this characterization in the following proposition.

\begin{proposition}\label{supernormal-properties} Let $\fkm\subset \fkn$ be nilpotent  $\FF_q$-algebras.  Then the following are equivalent:
\begin{enumerate}
 \item[(1)] $U_\fkm$ is supernormal in $U_\fkn$.
\item[(2)] $U_\fkm \vartriangleleft U_\fkn$ and $g X \in \fkm$ for all $g \in U_\fkn$ and $X \in \fkm$.
\item[(3)] $\fkm$ is a two-sided ideal in $\fkn$.

\item[(4)]  If $\fkm^\perp = \{ \gamma \in \fkn^* : \ker \gamma \supset \fkm\}$, then $U_\fkm \gamma U_\fkm = \{ \gamma \}$ for all $\gamma \in \fkm^\perp$.

\item[(5)] $g\gamma h \in \fkm^\perp$ for all $g,h \in U_\fkn$ and $\gamma \in \fkm^\perp$. 

\end{enumerate}
\end{proposition}

\begin{proof}
(1) $\Rightarrow$ (2) is obvious after noting that $1 + gXg^{-1}  = g(1+X)g^{-1}$.  (2) $\Rightarrow$ (3) since (2) implies that 
\[AX=(1+A)X-X  \in \fkm\qquad\text{and}\qquad XA = (1+A)^{-1}((1+A)X) (1+A) -X \in \fkm\] for all $A \in \fkn$ and $X \in \fkm$. 
(3) $\Rightarrow$ (4) since if $\fkm$ is a two-sided ideal, then $g\gamma(X) = \gamma(X) + \gamma((g^{-1}-1)X) = \gamma(X)$ and similarly $\gamma g(X) 
= \gamma(X)$ for all $\gamma \in \fkm^\perp$, $X \in \fkn$, and $g \in U_\fkm$.  
(4) $\Rightarrow$ (5) since if (4) holds and $g \in U_\fkn$, $\gamma \in \fkm^\perp$, and $X \in \fkm$, then
\[ g\gamma(X) = \gamma\((g^{-1}-1)X\) +{ \gamma(X)} = \bigr({\gamma(1+X)^{-1}}-\gamma\bigr)(g^{-1}-1)= 0 \] and similarly $\gamma g(X) = 0$.  Finally, to show that (5) $\Rightarrow$ (1), suppose $U_\fkm$ is not supernormal, so that $Y \overset{\mathrm{def}}= g^{-1}Xh^{-1} \notin \fkm$ for some $g,h \in U_\fkn$ and $X \in \fkm$.  Choose a complementary subspace $\fk c \subset \fkn$ with $\fkm \subset \fk c$ such that $\fkn = \FF_q Y \oplus \fk c$ as a vector space.  If we define $\gamma \in \fkn^*$ by $\gamma(t Y + C) = t$ for $t\in \FF_q$ and $C \in \fk c$, then $\gamma \in \fkm^\perp$ but $g\gamma h \notin \fkm^\perp$ since $g\gamma h(X) = \gamma(Y) =1$.  Hence (5) $\Rightarrow$ (1).
  \end{proof}

This result nicely characterizes supernormal subgroups given by $\FF_q$-subalgebras: namely, these correspond to ideals of $\fkn$.  We cannot classify \emph{arbitrary} supernormal subgroups as easily.   The problem here is with the field $\FF_q$, since if $q$ is not prime then some supernormal subgroups may come from algebras over a subfield of $\FF_q$ rather than $\FF_q$ itself.  However, we do have the following compromise.

\begin{proposition}
Write $q = p^a$ where $p$ is prime and $a$ is a positive integer, and let $\fkn$ be a nilpotent $\FF_q$-algebra.   Then every supernormal subgroup $H$ of $U_\fkn$ is an algebra group over the subfield $\FF_p \subset \FF_q$, in the sense that the set $\fkh = \{ X \in \fkn : 1+X \in H\}$ is a nilpotent $\FF_p$-algebra.
\end{proposition}

\begin{proof}
By Proposition \ref{sc-prop}, it suffices to show this when $H = \ker \chi_\fkn^\lambda$ for some $\lambda \in \fkn^*$.  In this case, (\ref{formula}) shows that  $\fkh = \bigcap_{\mu \in U_\fkn \lambda U_\fkn} \ker (\theta\circ \mu)$.  We must include $\theta$ here since if $q$ is not prime, then the nontrivial character $\theta$ on $\FF_q^+$ has a nontrivial kernel.  Clearly $\ker(\theta\circ \mu) = \{ X \in \fkn : \theta(\mu(X)) = 1\}$ is an additive group, however, so the same is true of $\fkh$ and this is enough to conclude that $\fkh$ is a $\FF_p$-vector space.  Consequently $\fkh$ is a nilpotent $\FF_p$-algebra since every $X \in \fkh \subset \fkn$ is nilpotent and
\[ X,Y \in \fkh\ \Rightarrow\ (1+X)(1+Y) = 1 + X + Y + XY \in H\ \Rightarrow\ X+Y+XY \in \fkh\ \Rightarrow\ XY \in \fkh.\] \end{proof}

Hence, if $q$ is prime then the supernormal subgroups of $U_\fkn$ are in bijection with the two-sided ideals in $\fkn$; specifically, every supernormal subgroup is then of the form $U_\fka$ for a two-sided ideal $\fka\subset \fkn$.  As a corollary, the proposition allows us to prove that the product of any two supernormal subgroups of an algebra group is supernormal.

\begin{corollary}
If $\fkn$ is a nilpotent $\FF_q$-algebra, then the product of any two supernormal subgroups of $U_\fkn$ is a supernormal subgroup.
\end{corollary}

\begin{proof}
Suppose $H,H' \subset U_\fkn$ are supernormal subgroups. Let $\fkh = \{ X \in \fkn : 1+X \in H\}$ and $\fkh' = \{ X + \fkn : 1 +X \in H'\}$, and write $\fk p = \{ X \in \fkn : 1+X \in HH'\}$.   Since $\fkh$ and $\fkh'$ are both additive groups by the proposition, we have $|\fkh + \fkh'| = \frac{|\fkh||\fkh'|}{|\fkh \cap \fkh'|} = \frac{|H||H'|}{|H\cap H'|} = |HH'| = |\fk p|$.  On the other hand, $\fk p \subset \fkh+\fkh'$ since if $1+X \in HH'$ then $X = Y+(1+Y)Y'$ for some $Y \in \fkh$ and $Y' \in \fkh'$, in which case $(1+Y)Y' \in \fkh'$ by supernormality.  It follows that $\fk p = \fkh + \fkh'$.  Consequently $HH'$ is a union of superclasses, since if $x \in HH'$ then $x = 1 + Y + Y'$ for some $Y \in \fkh$ and $Y' \in \fkh'$, so for all $g,h \in U_\fkn$, $1 + g(x-1)h = 1 + gYh + gY'h \in HH'$ since $gYh \in \fkh$ and $gY'h \in \fkh'$. 
\end{proof}

Analogous to the usual irreducible character theory, we have a notion of a supercharacter {lifted} from a quotient by a supernormal subgroup.  Suppose $\fkm \subset \fkn$ are nilpotent $\FF_q$-algebras with $U_\fkm$ supernormal in $U_\fkn$.  Then $\fkm$ is an ideal in $\fkn$, so the quotient $\fkn/\fkm$ is a well defined $\FF_q$-algebra.  For each supercharacter $\chi$ of the algebra group $U_{\fkn / \fkm}$, we define its \emph{lift} to $U_\fkn$ as the function $\wt \chi : U_\fkn \rightarrow \CC$ given by
\[ \wt\chi(1+X) = \chi\(1+(X+\fkm)\),\quad\text{for }1+X\in U_\fkn.\]  
This function is unsurprisingly a supercharacter of $U_\fkn$, a result which we state this formally as the following proposition.

\begin{proposition}\label{lift-prop}   Suppose $\fkm \subset \fkn$ are nilpotent $\FF_q$-algebras with $U_\fkm$ supernormal in $U_\fkn$.  If $\chi$ is a supercharacter of $U_{\fkn/\fkm}$, then its lift $\wt \chi$ is a supercharacter of $U_\fkn$ with the same degree as $\chi$.  Furthermore, $\wt \chi$ is irreducible if and only if $\chi$ is irreducible, and the map $\chi \mapsto \wt \chi$ gives a bijection between the set of supercharacters of $U_{\fkn/\fkm}$ and the subset of supercharacters of $U_{\fkn}$ whose kernels contain $U_\fkm$. 
\end{proposition}

\begin{proof}
We have a canonical isomorphism $\varphi : U_\fkn /U_\fkm \rightarrow U_{\fkn /\fkm} $ defined by $\varphi\((1+X)U_\fkm\) =1 + (X+\fkm)$ for $X \in \fkm$. This follows because $U_\fkn / U_\fkm$  and $U_{\fkn/\fkm}$ have the same order and because $\varphi$ is an injective homomorphism due to the fact that for all $X,Y \in \fkm$, 
\[ (1+X)(1+Y)^{-1} = 1 + (X-Y)(1+Y)^{-1} \in U_\fkm \ \Leftrightarrow\ (X-Y)(1+Y)^{-1} \in \fkm \ \Leftrightarrow\ X-Y \in \fkm.\]  
If $\chi$ is a supercharacter of $U_{\fkn/\fkm}$, then $\chi \circ \varphi$ is a character of $U_{\fkn}/U_{\fkm}$, and by definition the supercharacter lift $\wt \chi$ is equal to usual lift of the character $\chi\circ \varphi$ from the quotient group $U_\fkn / U_\fkm$ to $U_\fkn$.  

It follows from this observation and elementary properties of lifted characters (see \cite[Chapter 17]{JL}) that $\wt\chi$ is a character of $U_\fkn$ with the same degree as $\chi$ whose kernel contains $U_\fkm$, and that $\wt \chi$ is irreducible if and only if $\chi$ is irreducible.  To see that $\wt\chi$ is a supercharacter of $U_\fkn$, one can check that if $\chi = \chi_{\fkn/\fkm}^\lambda$ for some $\lambda \in (\fkn/\fkm)^*$ then $\wt \chi = \chi_\fkn^{\lambda'}$ where $\lambda' \in \fkn^*$ is defined by $\lambda'(X) = \lambda(X+\fkm)$.

If we extend the definition of $\wt \chi$ to any character $\chi$ of $U_{\fkn/\fkm}$, then the map $\chi \mapsto \wt \chi$ gives a linear bijection from the set of all characters of $U_{\fkn/\fkm}$ to the subset of characters of $U_\fkn$ whose kernels contain $U_\fkm$.  Since for an arbitrary character $\chi$ of $U_{\fkn/\fkm}$, $\wt \chi$ is a superclass function of $U_\fkn$ if and only if $\chi$ is a superclass of function of $U_{\fkn/\fkm}$, it follows that that $\chi \mapsto \wt \chi$ restricts to a bijection from the set of supercharacters of $U_{\fkn/\fkm}$ to the subset of supercharacters of $U_{\fkn}$ whose kernels contain $U_\fkm$. 
\end{proof}

\section{Supernormal Pattern Subgroups}\label{supernormal2-sect}

We can say something quite definite about when a pattern subgroup is supernormal in another pattern group, and results of this kind occupy the initial parts of this section.  Our main result, however, is the classification of all supernormal algebra subgroups in $U_n(q)$ when $q$ is prime, and this occupies the rest of the section.

To begin, we observe that for pattern groups, the definitions of normal and supernormal are equivalent.

\begin{proposition}\label{normal=supernormal}
 If $\cP \subset \cQ$ are posets with $U_\cP \vartriangleleft U_\cQ$, then $U_\cP$ is supernormal in $U_\cQ$.\end{proposition}

This result is Lemma 3.2 in \cite{MT09}, but we provide a short, direct proof below.  First, we translate the property $U_\cP\vartriangleleft U_\cQ$ into a condition on the posets $\cP\subset \cQ$.  If $\cP,\cQ$ are two posets on $[n]$ with $\cP\subset \cQ$, then we say that $\cP$ is \emph{normal} in $\cQ$, and write $\cP\vartriangleleft \cQ$, if the following conditions hold:
\begin{enumerate}
\item[(1)] If $(i,j) \in \cQ$ and $(j,k) \in \cP$ then $(i,k)\in \cP$.
\item[(2)] If $(j,k) \in \cQ$ and $(i,j) \in \cP$ then  $(i,k) \in \cP$.  
\end{enumerate}
Of course, our reason for adopting this notation has much to do with the following lemma.

\begin{lemma} \label{normal}
 If $\cP\subset \cQ$ are posets on $[n]$, then $U_\cP \vartriangleleft U_\cQ$ if and only if $\cP \vartriangleleft \cQ$.
\end{lemma}

\begin{proof}
 Suppose $\cP \vartriangleleft \cQ$ and $X \in \fkn_\cP$.  If $(i,j) \in \cQ$, then $e_{ij}X$ has nonzero entries only in the $i$th row, and $(e_{ij}X)_{ik} \neq 0 \Rightarrow X_{jk}\neq 0 \Rightarrow (j,k) \in \cP \Rightarrow (i,k) \in \cP$, so $e_{ij}X \in \fkn_\cP$.  Likewise, if $(j,k) \in \cQ$, then $Xe_{jk}$ has nonzero entries only in the $k$th column, and $(Xe_{jk})_{ik} \neq 0 \Rightarrow X_{ij} \neq 0\Rightarrow (i,j) \in \cP \Rightarrow (i,k) \in \cP$, so $Xe_{jk} \in \fkn_\cP$.  Hence if $x = 1+te_{ij} \in U_{\cQ}$ for some $(i,j) \in \cQ$ and $t\in \FF_q^\times$, then \[x(1+X)x^{-1} = 1 + X + te_{ij}X -t Xe_{ij} \in U_\cP.\]  Since elements of the form $x$ generate $U_\cQ$, $U_\cP \vartriangleleft U_\cQ$.

Now suppose $U_\cP \vartriangleleft U_\cQ$.  If $(j,k) \in \cP$ and $(i,j) \in \cQ$, then for $x=1+e_{ij} \in U_\cQ$ and $y=1+e_{jk} \in U_\cP$, we have $xyx^{-1} = 1+e_{jk} + e_{ik} \in U_\cP$ so $(i,k) \in \cP$.  Likewise, if $(i,j) \in\cP$ and $(j,k) \in \cQ$ then for $x = 1 + e_{jk} \in U_\cQ$ and $y= 1 + e_{ij} \in U_\cP$ we have $x^{-1}yx = 1 + e_{ij} + e_{ik} \in U_\cP$ so $(i,k)\in \cP$.  Hence $\cP\vartriangleleft \cQ$.
\end{proof}
%


Using this lemma, we have the following proof of Proposition \ref{normal=supernormal}.

\begin{proof}[Proof of Proposition \ref{normal=supernormal}]
 Given (2) of Proposition \ref{supernormal-properties}, it suffices to show that $gX \in \fkn_\cP$ for all $g \in U_\cQ$.  For this, we simply observe that if $(i,k) \notin\cP$ then for all $(i,j) \in \cQ$, we must have $(j,k)\notin\cP$ since $\cP \vartriangleleft \cQ$.  Hence if $(i,k) \notin\cP$ then $(gX)_{ik} = \sum_{j} g_{ij}X_{jk}= 0$ so $gX \in \fkn_\cP$.
\end{proof}

As a corollary, we have a correspondence between normal posets and nilpotent ideals.  Given a poset $\cQ$ on $[n]$, let $\fk t_\cQ$ denote the incidence algebra of $n\times n$ upper triangular matrices $X$ over $\FF_q$ such that $X_{ij} = 0$ if $i\neq j$ and $(i,j) \notin \cQ$.  In other words, $\fk t_\cQ$ is the algebra of $n\times n$ matrices over $\FF_q$ of the form $X= D + Y$ where $D$ is diagonal and $Y \in \fkn_\cQ$.  We now have the following result.

\begin{corollary} If $\cQ$ is a poset on $[n]$, then $\fka$ is a nilpotent two-sided ideal in $\fk t_\cQ$ if and only if $\fka = \fkn_\cP$ for some poset $\cP\vartriangleleft \cQ$.
\end{corollary}

\begin{remark}
This result characterizes normal pattern subgroups.  In odd characteristic, one can similarly characterize \emph{all} pattern subgroups as those subgroups of $\GL(n,\FF_q)$ invariant under conjugation by the subgroup of diagonal matrices; see \cite[Proposition 2.1]{DT09}.
\end{remark}

\begin{proof}
Proposition \ref{normal=supernormal} and Lemma \ref{normal} show that $\fkn_\cP$ is a nilpotent two-sided ideal in $\fk t_\cQ$ if $\cP\vartriangleleft \cQ$.  Conversely, suppose $\fka \subset \fk t_\cQ$ is a nilpotent two-sided ideal.  Then necessarily $\fka\subset \fkn_\cQ$, and if some $X \in \fka$ has $X_{jk} \neq 0$ for $(j,k) \in \cQ$ then $\fka \supset \FF_q\spanning\{e_{jk}\}$ since $te_{jk} = \frac{t}{X_{jk}} e_{jj} X e_{kk} \in \fka$.  Furthermore, if $(i,j) \in \cQ$ and $e_{jk} \in \fka$ or  if $(j,k) \in \cQ$ and $e_{ij} \in \fka$ then $e_{ik} \in \fka$, since $e_{ik} = e_{ij}e_{jk} \in \fka$.  It follows that the set $\cP = \{ (i,j) \in\cQ : \exists X \in \fka\text{ with }X_{ij} \neq 0\}$ is a poset with $\cP\vartriangleleft\cQ$, and that $\fka = \fkn_\cP$.
\end{proof}

This result and \cite[Proposition 2]{S75} show that the number of normal subposets $\cP\vartriangleleft [[n]] $ is the $n$th Catalan number $C_n = \frac{1}{n+1}\binom{2n}{n}$.  We will see this directly later.  More generally, we can easily classify all normal subposets of a given poset $\cP$.  To do this, we define a strict partial ordering $\prec_\cP$ of $\cP$ itself by setting
\be\label{poset-partial-order}\ba  (j,k) &\prec_\cP (i,k)&\text{ iff }(i,j) \in \cP \\
(j,k) &\prec_\cP (j,l)&\text{ iff }(k,l) \in \cP
\ea\qquad\text{for $(j,k),(i,k), (j,l) \in \cP$,}
\ee and extending transitively (see Section 3.1 of \cite{DT09} for more details).  For example, if $\cP = [[n]]$ then $(j,k) \preceq_{[[n]]} (i,l)$ if and only if $i\leq j < k \leq l$.

Now suppose $\cP \vartriangleleft \cQ$ are posets.  If $(i,k) \in \cQ - \cP$, then  $(j,k) \notin\cP$ whenever $(i,j) \in \cQ$ and $(i,j) \notin \cP$ whenever $(j,k) \in \cP$.  Comparing this with (\ref{poset-partial-order}), we see that if $(k,l) \in \cQ - \cP$, then $(i,j)\notin \cP$ for all $(i,j) \prec_\cQ (k,l)$.  Formalizing this observation gives the following lemma.

\begin{lemma}\label{normal2partial-order} If $\cP \subset \cQ$ are posets, then $\cP \vartriangleleft \cQ$ if and only if $(k,l) \in \cQ-\cP$ implies $(i,j) \notin \cP$ for all $(i,j) \preceq_\cQ (k,l)$.
\end{lemma}

\begin{proof}
The preceding discussion gives the forward direction.  For the converse, suppose the condition in the lemma holds.  If $(i,j) \in \cQ$ and $(j,k) \in \cP$, then $(i,k) \in \cQ$ and $(j,k) \preceq_\cQ(i,k)$, so $(i,k) \in \cP$.  If $(j,k) \in \cQ$ and $(i,j) \in \cP$, then $(i,k) \in \cQ$ and $(i,j) \preceq_\cQ(i,k)$, so again $(i,k) \in \cP$.  Hence $\cP\vartriangleleft \cQ$ by definition.
\end{proof}

  This lemma leads to our classification.  Recall that a subset of a partially ordered set is an \emph{antichain} is no two elements of the set are comparable.
  
\begin{proposition}\label{normal-pattern-classification}
 Let $\cP$ be a poset on $[n]$.  Then the set of normal subposets of $\cP$ is in bijection with the set of subsets $S\subset \cP$ which are antichains with respect to the partial ordering $\prec_\cP$.  This bijection is given explicitly by $S \mapsto \cP_S$, where $\cP_S$ is the poset on $[n]$ defined by \[\cP_S = \{ (i,j) \in \cP: (i,j)\not \preceq_\cP(k,l)\text{ for all }(k,l) \in S\}.\] 
\end{proposition}

\begin{proof}
 Observe that if $(i,j),(j,k) \in \cP$ then $(i,j) \prec_\cP (i,k)$ and $(j,k) \prec_\cP (i,k)$.   Since $\prec_\cP$ is transitive, it follows that $\cP_S$ is a poset on $[n]$ for any subset $S\subset \cP$, 
 and by Lemma \ref{normal2partial-order} we have $\cP_S \vartriangleleft \cP$. 
 Conversely, suppose $\cO$ is a poset with $\cO\vartriangleleft \cP$.  If we define $S \subset \cP$ to be the set of maximal elements in $\cP-\cO$ with respect to $\prec_\cP$, then $S$ is an antichain and $\cO = \cP_S$ by definition.  Furthermore, $S$ is clearly the only antichain in $\cP$ with $\cO = \cP_S$, so it follows that the map $S\mapsto \cP_S$ between antichains and normal subposets of $\cP$ is a bijection.
 \end{proof}

\def\cD{\mathcal{D}}

The rest of this section concerns the problem of counting and classifying the set of supernormal algebra subgroups of $U_n(q)$. 
In order to do this, we recall some familiar definitions from combinatorics.  In particular, a \emph{Dyck path} of order $n$ is a lattice path in the plane consisting only of up steps $U = (1,1)$ and down steps $D = (1,-1)$,
 which starts at $(0,0)$, ends at $(n,0)$, and never passes below the $x$-axis.  Let $\cD_n$ denote the set of such paths.  We can uniquely represent each $\rho \in \cD_n$ by writing $\rho = U^{a_1} D^{b_1} \cdots U^{a_r} D^{b_r}$ where $a_i,b_i$ are positive integers such that $a_1 + \dots + a_r = b_1 + \dots + b_r = n$ and $(a_1 - b_1) + \dots + (a_t - b_t) \geq 0$ for all $t=1,\dots,r$.  This notation indicates that $\rho$ is the path given by taking $a_1$ up steps, then $b_1$ down steps, then $a_2$ up steps, then $b_2$ down steps, and so on.  
 The number of peaks of $\rho \in \cD_n$ is the number of occurrences of an up step $U$ following consecutively by a down step $D$.  For example, the path $\rho = U^{a_1} D^{b_1} \cdots U^{a_r} D^{b_r}$ has $r$ peaks.  It is well-known that the number of Dyck paths of order $n$ with $r$ peaks is the Narayana number $N(n,r)$, which is defined by 
\[ N(n,k+1) = \left\{\barr{ll} \displaystyle \frac{1}{n}  \binom{n}{k+1}\binom{n}{k}, & \text{if }0\leq k < n, \\ 0,&\text{if $k\geq n$.}\earr\right.\]  The following lemma gives an additional interpretation of Narayana numbers.

\begin{lemma}\label{count1} Fix a positive integer $n$ and nonnegative integer $k$.  Then the number of $k$-element subsets of $[[n]] = \{ (i,j) : 1\leq i < j \leq n\}$ which are antichains with respect to $\prec_{[[n]]}$ is $N(n,k+1)$.
\end{lemma}

\begin{proof}  
Fix a $k$-element antichain $S\subset [[n]]$, and write $S = \{(i_1,j_1),\dots, (i_k,j_k)\}$ where $i_1 < \dots < i_k$; we can enumerate $S$ in this way since no two positions in $S$ lie the same row.  Since $S$ is an antichain, it follows in addition that $j_1 < \dots < j_k$.  Now define $\rho_S \in \cD_n$ by
\[ \rho_S = \left\{\barr{ll} 
U^{j_1-1} D^{i_1} \(U^{j_2-j_1} D^{i_2-i_1} \cdots U^{j_k-j_{k-1}} D^{i_k-i_{k-1}} \)U^{n-(j_k-1)} D^{n-i_k}, & \text{if }k>0,
 \\ 
 U^nD^n,& \text{if }k=0.\earr
 \right.\] 
 By construction all of the exponents here are positive integers, and since $i_t < j_t$ for all $t$, the path stays above the $x$-axis and lies in $\cD_n$.  Furthermore, $\rho_S$ has $k+1$ peaks.  To prove the lemma it suffices to show that $S \mapsto \rho_S$ is bijective as a map from the set of $k$-elements antichains in $[[n]]$ to the set of Dyck paths of order $n$ with $k+1$ peaks.  Indeed, this follows since we can construct an inverse map: given $\rho =  U^{a_1} D^{b_1} \cdots U^{a_{k+1}} D^{b_{k+1}} \in \cD_n$, let 
 \[ S_\rho = \left\{\barr{ll} \{ (i_1,j_1),\dots, (i_k,j_k)\}\ \ \text{where}\ \ \left\{ \barr{l} i_t = b_1 + \dots + b_t, \\ j_t = 1 + a_1 +\dots + a_t, \earr\right. &\text{if }k>0, \\ \varnothing, & \text{if }k=0.\earr\right.\]
Clearly $i_1<\dots <i_k$ and $j_1<\dots < j_k$, and the condition $(a_1 - b_1) + \dots + (a_t - b_t) \geq 0$ for all $t$ implies that $i_t < j_t$ for all $t$.  Thus $S_\rho \subset [[n]]$ is an antichain, and one can check that $S\mapsto \rho_S$ and $\rho \mapsto S_\rho$ are inverse maps, and therefore both bijections, as desired.
\end{proof}

The preceding proof  constructs a bijection between antichains in $[[n]]$ and Dyck paths of order $n$.  We can visualize this bijection by viewing the positions of an $n\times n$ matrix as the interior squares of an $n\times n$ grid. We then rotate our Dyck paths 45 degrees clockwise to view $\cD_n$ as the set of monotonic paths consisting of moves to the right and down along edges in the grid, which start at the upper left hand corner, end at the lower right corner, never pass below the diagonal.  The Dyck path corresponding to an antichain $S\subset [[n]]$ is the unique monotonic path whose valleys (i.e., occurrences of a move down followed consecutively by a move right) border positions of $S$.  For example,
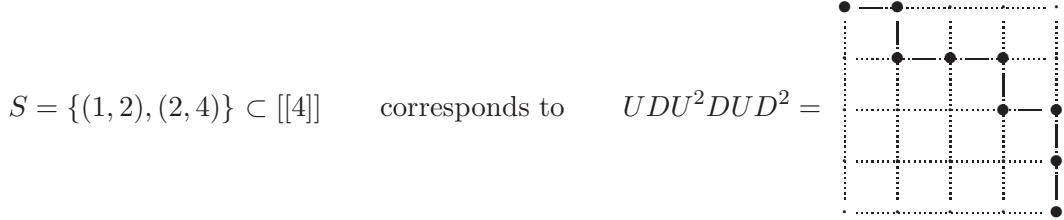
\[ S = \{ (1,2), (2,4) \} \subset [[4]]  \qquad\text{corresponds to}\qquad 
UDU^2DUD^2 = \xy<0.45cm,1.35cm> \xymatrix@R=.3cm@C=.3cm{
\bullet \ar @{-} [r] \ar @{.} [rrrr]& \bullet   & \cdot  &\cdot  &\cdot  \\
\cdot   \ar @{.} [rrrr]& \bullet \ar @{-} [u] \ar @{-} [r]  & \bullet \ar @{-} [r]  & \bullet  &\cdot \\
\cdot \ar @{.} [rrrr]  & \cdot & \cdot & \bullet \ar @{-} [u] \ar @{-} [r] & \bullet   \\
\cdot \ar @{.} [rrrr]  & \cdot & \cdot & \cdot &    \bullet \ar @{-} [u] \\
\cdot \ar @{.} [rrrr] \ar @{.} [uuuu]  & \cdot \ar @{.} [uuuu]& \cdot \ar @{.} [uuuu]& \cdot \ar @{.} [uuuu]& \bullet\ar @{.} [uuuu] \ar @{-} [u]  
}\endxy\] This shows directly that the number of normal subposets $\cP\vartriangleleft [[n]]$ is $C_n = \frac{1}{n+1}\binom{2n}{n}$, since \[|\cD_n| = \sum_{k=1}^n N(n,k) = C_n.\]

\def\RR{\mathbb{R}}

\def\cU{\mathcal{U}}
\def\cV{\mathcal{V}}
\def\cW{\mathcal{W}}

For the second step in our classification, we present another set of definitions.  Given a nonnegative integer $k$, let $\cV_k(\FF_q)$ denote the set of all subspaces of the $k$-dimensional vector space $\FF_q^k$ over $\FF_q$.   The number of subspaces of $\FF_q^k$ of dimension $i$ is well known to be given by the $q$-binomial coefficient
\[\label{q-binom} \binom{k}{i}_q = \left\{\barr{ll} 1, & \text{if }i = 0,\\ \\ \displaystyle \frac{(1-q^k)(1-q^{k-1})\cdots (1-q^{k-i+1})}{(1-q)(1-q^2)\cdots(1-q^i)},& \text{if }1\leq i \leq k,\\ \\ 0,& \text{otherwise}.\earr\right.\] Thus, $\cV_k(\FF_q)$ has cardinality 
\[\label{V-card} \left|\cV_k(\FF_q)\right| = \sum_{i=0}^k \binom{k}{i}_q.\] The vector space
$\FF_q^k$ has a canonical basis given by the vectors $e_i = (0,\dots,0,1,0,\dots,0)$ for $1\leq i \leq k$ which have a single 1 in the $i$th coordinate and 0 in all other coordinates.  Define $\wt \cV_k(\FF_q)$ for $k>0$ as the set of subspaces of $\FF_q^k$ whose intersection with this canonical basis is empty, and set $\wt \cV_0(\FF_q) = \cV_0(\FF_q) = \{ \{ 0\}\}$.

\begin{lemma}\label{count2}
The cardinality of $\wt \cV_k(\FF_q)$ is given by the binomial transform of $\left|\cV_k(\FF_q)\right|$; i.e.,
\[\label{tilde-card} \left| \wt \cV_k(\FF_q)\right| = \sum_{j=0}^k (-1)^{k-j} \binom{k}{j} \left|\cV_j(\FF_q)\right| =\sum_{0\leq i \leq j \leq k}  (-1)^{k-j}\binom{k}{j} \binom{j}{i}_q.\]  
\end{lemma}

\begin{remark}
When $q=2$, the cardinalities $\left|\cV_k(\FF_q)\right|$ and $\left|\wt \cV_k(\FF_q)\right|$ for $k\geq 0$ appear as sequences \htmladdnormallink{A006116} {http://oeis.org/A006116}  and \htmladdnormallink{A135922} {http://oeis.org/A135922}  in \cite{OEIS}, respectively.
\end{remark}

\begin{proof} 
This holds if $k=0$; assume $k>0$.  Then $\left|\wt \cV_k(\FF_q)\right|  = \left|\cV_k(\FF_q)\right| - |\cT|$, where $\cT\subseteq \cV_k(\FF_q)$ denotes the set of subspaces of $\FF_q^k$ containing at least one basis vector.  Given a nonempty subset $S \subseteq \{1,\dots,k\}$, let $\cT_{S} \subseteq \cT$ denote the set of subspaces which contain $e_i$ for each $i \in S$.  For any nonempty $S$ we have $|\cT_{S}| = \left| \cV_{k-|S|}(\FF_q)\right|$, since if $\cW = \FF_q\spanning\{ e_{i} : i \in S\}$ then we can naturally identify  $\FF_q^k / \cW$ with $\FF_q^{k-|S|}$, and under this identification the map $\cU\mapsto \cU/\cW$ gives a bijection $\cT_S \rightarrow \cV_{k-|S|}(\FF_q)$.  Furthermore, if $R,S\subseteq \{1,\dots,k\}$ are two nonempty subsets, then $\cT_R \cap \cT_S = \cT_{R\cup S}$.  Since $\cT = \bigcup_{i=1}^k \cT_{\{i\}}$, the desired result follows from the inclusion-exclusion principle.
%
\end{proof}

Consider a $k$-element antichain $S \subset [[n]]$, and enumerate its elements as $(i_1,j_1),\dots,(i_k,j_k)$.  The choice of enumeration is not important, but to be canonical we can demand that $i_1 < \dots < i_k$ and $j_1< \dots < j_k$.   This is possible since otherwise two elements of $S$ would be comparable.  We define a linear map $\varphi_{S} : \fkn_n(q) \rightarrow \FF_q^k$ by 
\[\label{varphi_S} \varphi_S(X)  = \left\{\barr{ll} (X_{i_1j_1}, \dots, X_{i_kj_k}) \in \FF_q^k, & \text{if }k>0 \\ 0,&\text{if }k=0\earr\right.\qquad\text{for }X \in \fkn_n(q).\] 
Now, given a subspace $\cU\in \wt \cV_k(\FF_q)$, we define the subset $G_n(S,\cU) \subset U_n(q)$ by \be\label{group-def} G_n(S,\cU) = \left\{ 1 +X \in U_n(q) : \varphi_S(X) \in \cU\text{ and }X_{ij} = 0\text{ if }\exists (k,l) \in S\text{ with }(i,j)\prec_{[[n]]} (k,l)
\right\}.
\ee
The set $G_n(S,\cU)$ is in fact a supernormal subgroup of $U_n(q)$, and every supernormal subgroup of $U_n(q)$ is of this form when $q$ is prime.  This is the main result of this section, and we state it as the following theorem.

\begin{theorem}\label{supernormal-classification} If $q$ is prime, then the map
\be \label{supernormal2paths}
\barr{ccc} 
\left\{ (S, \cU) : \barr{l} \text{$S\subset [[n]]$ and $\cU \in \wt \cV_k(\FF_q)$, where $S$ is a} \\ \text{$k$-element antichain with respect to $\prec_{[[n]]}$} \earr \right\}
&
\to
&
\left\{\barr{c} \text{Supernormal} \\ \text{subgroups of $U_n(q)$} \earr\right\} \\
 (S,\cU) & \mapsto & G_n(S,\cU)
\earr
\ee is a bijection.  Consequently, the number of supernormal subgroups of $U_n(q)$ or two-sided ideals of $\fkn_n(q)$ is 
\[\label{number-supernormal-formula}\sum_{k=0}^{n-1} N(n,k+1) \left|\wt \cV_k(\FF_q)\right|
= \sum_{0\leq i \leq j \leq k < n}  \frac{(-1)^{k-j}}{n} \binom{n}{k+1} \binom{n}{k}\binom{k}{j} \binom{j}{i}_q.\] \end{theorem}

\begin{remark}
We showed above that the number of supernormal \emph{pattern} subgroups of $U_n(q)$ is the $n$th Catalan number  
\[ C_n =\sum_{k=0}^{n-1} N(n,k+1) = \frac{1}{n+1}\binom{2n}{n},\] which has no dependence on $q$.  By contrast, the theorem asserts that the number of arbitrary supernormal subgroups of $U_n(q)$ is strongly dependent on the size of the ambient field $\FF_q$.  When $q$ is prime $U_n(q)$ has a supernormal subgroup which is not a pattern group for all $n>2$.  If $q$ is not prime, then even $U_2(q) \cong \FF_q^+$ has a supernormal subgroup which is not a pattern group, given by the kernel of any nontrivial supercharacter.  
\end{remark}

To prove this, we first note the following characterization of the kernels of supercharacters of $U_n(q)$.  Recall that we define $(j,k) \prec_{[[n]]} (i,l)$ iff $i \leq j < k \leq l$  and one of the inequalities is strict.  The next lemma is now immediate from (\ref{added-formula}).

\begin{lemma}\label{classifying-lemma} Fix a positive integer $n$ and let $\lambda \in \sP_n^*(q)$.  Then the kernel of the supercharacter $\chi^\lambda$ of $U_n(q)$ is the subgroup
\[ \ker \chi^\lambda = \left\{ 1+X \in  U_n(q) :  \theta\circ \lambda(X) = 1\text{ and }X_{ij} = 0\text{ if }\exists (k,l) \in \supp(\lambda)\text{ with }(i,j)\prec_{[[n]]} (k,l)
\right\}.\]
\end{lemma}

We can now prove the theorem.

\begin{proof}[Proof of Theorem \ref{supernormal-classification}]
Our goal is to prove that the map (\ref{supernormal2paths}) is a bijection.  To begin, let us show that this map is well defined; specifically, we must demonstrate that (\ref{group-def}) defines a supernormal subgroup.  Let $S \subset [[n]]$ be a $k$-element antichain and let $\cU \in \wt \cV_k(\FF_q)$ be a subspace of dimension $k-r$.  If $S =\varnothing$ then necessarily $\cU = \{0\}$ and $G_n(S,\cU) = U_n(q)$ is a supernormal subgroup.  Assume $k>0$ and choose functionals $\ell_1,\dots, \ell_r \in (\FF_q^k)^*$ such that $\cU = \ker \ell_1 \cap \dots \cap \ker \ell_r$.  
Now define $\lambda_i \in \fkn_n^*(q)$ by 
\[ \lambda_i(X) =  \ell_i\circ \varphi_S(X),\qquad\text{for }X \in \fkn_n(q).\]  Since $S$ is an antichain, each $\lambda_i \in \sP_n^*(q)$, and since $\cU$ contains none of the basis vectors $e_1,\dots, e_k \in \FF_q^k$, we have $\supp(\lambda_1)\cup \dots\cup \supp(\lambda_r) = S$.  Since the nontrivial homomorphism $\theta : \FF_q^+ \rightarrow \CC^\times$ is injective as $q$ is prime, we also have $\theta\circ \lambda_i(X) = 0$ for all $i$ if and only if $X \in \ker\lambda_i$ for all $i$, in which case $\varphi_S(X) \in \ker \ell_1 \cap \dots \cap \ker \ell_r = \cU$.  It now follows from Lemma \ref{classifying-lemma} that the set $G_n(S,\cU)$ defined by (\ref{group-def}) is in fact the supernormal subgroup 
$G_n(S,\cU) = \ker \chi^{\lambda_1}\cap \dots \cap \ker \chi^{\lambda_r}.$  

This shows that (\ref{supernormal2paths}) is well defined.  To see that the map is injective, observe that \[\cU = \varphi_S\( G_n(S,\cU) - 1\) = \varphi_S\( \{ X \in \fkn_n(q): 1+X \in G_n(S,\cU) \}\);\] hence if $S$  is a fixed $k$-element antichain and $\cU,\cW \in \wt\cV_k(\FF_q)$, then $G_n(S,\cU) = G_n(S,\cW)$ if and only if $\cU = \cW$.  Now suppose $S,T$ are antichains in $[[n]]$ with $S\neq T$; let $\cU \in \wt \cV_{k_1}(\FF_q)$ and $\cW \in \wt \cV_{k_2}(\FF_q)$ where $k_1 = |S|$ and $k_2 = |T|$.  One of $S-T$, $T-S$ must be nonempty, so assume without loss of generality that $(i,j) \in S - T$.  Then one of the following three cases must occur: 
\begin{enumerate}
\item[(i)] $\{(i,j)\} \cup T$ is an antichain.  Then $1+e_{ij} \in G_n(T,\cW) - G_n(S,\cU)$.  
\item[(ii)] $(i,j) \prec_{[[n]]} (i',j')$ where $(i',j') \in T$. Then $1+e_{i'j'} \in G_n(S,\cU) - G_n(T,\cW)$.   
\item[(iii)] $(i',j') \prec_{[[n]]} (i,j)$ where $(i',j') \in T$. Then $1+e_{ij} \in G_n(T,\cW) - G_n(S,\cU)$. 
\end{enumerate}
In every case $G_n(S,\cU)\neq G_n(T,\cW)$.  We therefore conclude that (\ref{supernormal2paths}) is injective.

Finally, to show our map's surjectivity, consider an arbitrary supernormal subgroup $G\subset U_n(q)$ of the form 
$G = \ker \chi^{\lambda_1}\cap \dots \cap \ker \chi^{\lambda_r}$, where each $\lambda_i \in \sP_n^*(q)$.  Let $S$ be the set of positions in $\supp(\lambda_1)\cup \dots \cup \supp(\lambda_r)$ which are maximal with respect to $\prec_{[[n]]}$, and set $k = |S|$.  Then $S$ is an antichain, and if $1+X \in G$ then $X_{ij} = 0$ whenever $(i,j)\prec_{[[n]]} (i',j')$ for some $(i',j') \in S$ by Lemma \ref{classifying-lemma}.  Therefore setting $(\lambda_1)_{ij}  = \dots = (\lambda_r)_{ij} = 0$ for all $(i,j) \notin S$ has no effect on $G$, so we may assume without loss of generality that $\varnothing \subsetneq \supp(\lambda_i) \subset S$ for all $i$.  Let $\cU' = \ker(\lambda_1)\cap \dots \cap \ker (\lambda_r)$ and $\cU = \varphi_S(\cU') \in \cV_k(\FF_q)$.  Since every position in $S$ is in some $\supp(\lambda_i)$, no $X \in \cU'$ has $|\supp(X) \cap S| =1$; in other words, $\cU \in \wt \cV_k(\FF_q)$.  Since $\theta$ is injective as $q$ is prime, it follows from Lemma \ref{classifying-lemma} that $G = G_n(S,\cU)$.   This establishes the bijection (\ref{supernormal2paths}), and applying Lemmas \ref{count1} and \ref{count2} gives the formula for the number of supernormal subgroups of $U_n(q)$.
\end{proof}

\section{Restriction to a Supernormal Subgroup}\label{restrict-sect}

In this section we investigate the restriction of a supercharacter of an algebra group to a supernormal algebra subgroup.  Theorem \ref{resind-thm} computes how such restrictions decompose in general, and Proposition \ref{codim1-prop} provides a more explicit description for subgroups corresponding to subalgebras $\fkm \subset \fkn$ of codimension one.  

Suppose $\fkm \subset \fkn$ are  nilpotent associative $\FF_q$-algebras with $U_\fkm$ supernormal in $U_\fkn$.  Then by Proposition \ref{supernormal-properties} we have commuting left and right actions of $U_\fkn$ on $\fkm$ by multiplication, and these actions give rise to commuting left and right actions of $U_\fkn$ on $\fkm^*$.  Since $gU_\fkm = gU_\fkm g^{-1} g = U_\fkm g$ for all $g \in U_\fkn$, we can view $U_\fkn$ as acting (on the left and right) on the left, right, and two-sided $U_\fkm$ orbits of $\fkm$ and $\fkm^*$.  For example we have 
\[ gU_\fkm X U_\fkm h  = U_\fkm(gXh)U_\fkm\qquad\text{and}\qquad g U_\fkm \lambda U_\fkm h = U_\fkm (g\lambda h) U_\fkm\] for $g,h \in U_\fkn$, $X \in \fkm$, $\lambda \in \fkm^*$.  These actions evidently preserve all orbit sizes, so it follows that each left/right/two-sided $U_\fkn$-orbit in $\fkn$ or $\fkn^*$ decomposes as a disjoint union of left/right/two-sided $U_\fkm$-orbits, all of which have the same cardinality.

\begin{remark}
When dealing with the action of $U_\fkn$ on $\fkm^*$, we are careful to distinguish linear functionals by their domains.  To avoid ambiguity, we never implicitly identify $\lambda \in \fkn^*$ with a linear functional on $\fkm$; instead, we will always denote the identification explicitly by writing $\lambda \downarrow \fkm \in \fkm^*$.  Thus, for $\lambda \in \fkn^*$, the orbit $U_\fkn \lambda U_\fkn$ lies in $\fkn^*$, while the orbit $U_\fkn(\lambda \downarrow \fkm)U_\fkn$ lies in $\fkm^*$.
\end{remark}

We can view $U_\fkn$ as acting directly on the sets of superclasses and supercharacters of $U_\fkm$ on the left and right by 
\[\label{character-action-def} \ba  x\cdot \cK_\fkm^g\cdot y &\overset{\mathrm{def}}= \{ 1 + x(h-1)y : h \in \cK_\fkm^g \} = \cK_\fkm^{1 + x(g-1)y}, \\
 x \chi_\fkm^\lambda y(g) &\overset{\mathrm{def}}= \chi_\fkm^\lambda(1+x^{-1}(g-1)y^{-1}) = \chi_\fkm^{x\lambda y}(g),
\ea \qquad\text{for }x,y \in U_\fkn,\ g \in U_\fkm,\ \lambda \in \fkm^*.
\] Of course we have $x\cdot \cK_\fkm^g \cdot y =  \cK_\fkm^g$ and $x\chi_\fkm^\lambda y = \chi_\fkm^\lambda$ if $x,y \in U_\fkm$, so these actions pass to a two-sided action of the quotient group $U_\fkn / U_\fkm$.  Since $U_\fkn$ preserves orbit sizes in $\fkm^*$, it follows that the $U_\fkn$-action on the supercharacters of $U_\fkm$ preserves both degree and inner products, in the sense that 
\be \label{character-action-observation}\ba  x  \chi_\fkm^\lambda  y(1) &=   \chi_\fkm^{x\lambda  y}(1) = \chi_\fkm^\lambda(1), \\
 \langle x  \chi_\fkm^\lambda  y, x  \chi_\fkm^\mu  y \rangle_{U_\fkm} &=
\langle   \chi_\fkm^{x \lambda  y},   \chi_\fkm^{x \mu y} \rangle_{U_\fkm} =  
\langle \chi_\fkm^\lambda, \chi_\fkm^\mu  \rangle_{U_\fkm},\ea \qquad\text{for }x,y\in U_\fkn,\ \lambda,\mu \in \fkm^*.
\ee  

We use these observations below to provide a supercharacter analogue for Clifford's theorem, a classical result which states that the restriction of an irreducible character to a normal subgroup decomposes as a sum of irreducible characters with the same degree and multiplicity.

\begin{theorem}\label{resind-thm}
 Let $\fkm \subset \fkn$ be  nilpotent associative $\FF_q$-algebras with $U_\fkm$ supernormal in $U_\fkn$.  \begin{enumerate}
 \item[(1)] Choose $\lambda \in \fkn^*$ and let $\mu = \lambda \downarrow \fkm \in \fkm^*$.  Then
\[\label{restrict-formula} \chi_\fkn^\lambda \downarrow U_\fkm = \frac{|U_\fkn \lambda| | U_\fkm \mu U_\fkm|}{|U_\fkm \mu| | U_\fkn\mu  U_\fkn|} \sum_\nu \chi_\fkm^\nu\] where the sum is over a set of representatives $\nu \in \fkm^*$ of the distinct two-sided $U_\fkm$-orbits in $U_\fkn\mu U_\fkn$.  Hence $\chi_\fkn^\lambda \downarrow U_\fkm$ decomposes as a sum of $\frac{|U_\fkn \mu U_\fkn|}{|U_\fkm \mu U_\fkm|}$ distinct supercharacters of $U_\fkm$ with the same degree and multiplicity.

\item[(2)] Choose $\mu \in \fkm^*$.  Then
\[\label{sind-formula} \SInd_{U_\fkm}^{U_\fkn} \(\chi_\fkm^\mu \) = \sum_\lambda \frac{|U_\fkm \mu||U_\fkn \lambda U_\fkn|}{|U_\fkn \lambda| |U_\fkn \mu U_\fkn|} \chi_\fkn^\lambda\] where the sum is over a set of representatives $\lambda \in \fkn^*$ of the distinct two-sided $U_\fkn$-orbits in $\fkn^*$ which on restriction to $\fkm$ are equal to $U_\fkn \mu U_\fkn$. 
\end{enumerate}
\end{theorem}

\begin{proof}
We first prove (2).  Observe that if 
$\lambda \in \fkn^*$, then $g(\lambda \downarrow \fkm)h = g\lambda h \downarrow \fkm$ for all $g,h \in U_\fkn$, since 
 \[ (g\lambda  h \downarrow \fkm)(X) = (g\lambda h)(X) =  \lambda(g^{-1}X h^{-1}) = (\lambda\downarrow \fkm)(g^{-1}X h^{-1}) = \(g(\lambda\downarrow \fkm) h\)(X),\qquad\text{for }X \in \fkm.\] 
 Consequently, from Lemma \ref{sind-lemma} we see that $\chi^\lambda_\fkn$ appears as a constituent of $\SInd_{U_\fkm}^{U_\fkn}(\chi_\fkm^\mu)$ if and only if $U_\fkn \lambda U_\fkn$ is equal to $U_\fkn \mu U_\fkm$ on restriction to $\fkm$, and in this case the number of elements of $U_\fkn \lambda U_\fkn$ which restrict to elements of $U_\fkm \mu U_\fkm$ is equal to $|U_\fkn \lambda U_\fkn|$ divided by $\frac{|U_\fkn \mu U_\fkn|}{|U_\fkm \mu U_\fkm|}$, the number of two-sided $U_\fkm$-orbits in $U_\fkn \mu U_\fkn$.  Thus by Lemma \ref{sind-lemma} we have 
 \[ \ba \SInd_{U_\fkm}^{U_\fkn}(\chi_\fkm^\mu) &= \sum_{\substack{\lambda \in \fkn^*\\ \lambda \downarrow \fkm  \in U_\fkm \mu U_\fkm}} \frac{|U_\fkm \mu|}{|U_\fkm \mu U_\fkm| |U_\fkn \lambda|}  \chi_\fkn^\lambda
  = \frac{|U_\fkn \lambda U_\fkn|}{|U_\fkn \mu U_\fkn| / |U_\fkm \mu U_\fkm|} \sum_{\lambda} \frac{|U_\fkm \mu|}{|U_\fkm \mu U_\fkm| |U_\fkn \lambda|}  \chi_\fkn^\lambda
\\&  = \sum_\lambda \frac{|U_\fkm \mu||U_\fkn \lambda U_\fkn|}{|U_\fkn \lambda| |U_\fkn \mu U_\fkn|} \chi_\fkn^\lambda\ea\] where the last two sums are over a set of representatives $\lambda \in \fkn^*$ of the distinct two-sided $U_\fkn$-orbits in $\fkn^*$ which on restriction to $\fkm$ are equal to $U_\fkn \mu U_\fkn$.  

To prove (1), we observe that if $\lambda \in \fkn^*$ has $ \lambda \downarrow \fkm = \mu$ and $\nu \in \fkn^*$, then by reciprocity
\[  \left\langle \chi_\fkm^\nu, \chi_\fkn^\lambda \downarrow U_\fkm \right\rangle_{\fkm} 
=
 \left\langle \SInd_{U_\fkm}^{U_\fkn}\(\chi_\fkm^\nu\), \chi_\fkn^\lambda \right\rangle_{U_\fkn}
 =
\left\{
\barr{ll} \displaystyle\frac{|U_\fkm \nu||U_\fkn \lambda U_\fkn|}{|U_\fkn \lambda| |U_\fkn \nu U_\fkn|} \left\langle \chi_\fkn^\lambda, \chi_\fkn^\lambda \right\rangle_{U_\fkn} ,& \text{if }\nu \in  U_\fkn \mu U_\fkn,
\\
0,& \text{otherwise}.\earr\right.\] Hence 
\[  \chi_\fkn^\lambda \downarrow U_\fkm = \sum_\nu \frac{|U_\fkm \nu||U_\fkn \lambda U_\fkn|}{|U_\fkn \lambda| |U_\fkn \nu U_\fkn|} \frac{\left\langle \chi_\fkn^\lambda, \chi_\fkn^\lambda \right\rangle_{U_\fkn}}{\left\langle \chi_\fkm^\nu, \chi_\fkm^\nu \right\rangle_{U_\fkm}} \chi_\fkm^\nu =
  \frac{|U_\fkm \mu||U_\fkn \lambda U_\fkn|}{|U_\fkn \lambda| |U_\fkn \mu U_\fkn|} \frac{\left\langle \chi_\fkn^\lambda, \chi_\fkn^\lambda \right\rangle_{U_\fkn}}{\left\langle \chi_\fkm^\mu, \chi_\fkm^\mu \right\rangle_{U_\fkm}} \sum_\nu\chi_\fkm^\nu \]
by (\ref{character-action-observation}), where both sums are over a set of representatives $\nu \in \fkm^*$ of the distinct two-sided $U_\fkm$-orbits in $U_\fkn\mu U_\fkn$.  (1) now follows from the observation that  
$\frac{|U_\fkn \lambda|}{\left\langle \chi_\fkn^\lambda, \chi_\fkn^\lambda \right\rangle_{U_\fkn}} = \frac{|U_\fkn \lambda U_\fkn|}{|U_\fkn \lambda|}$
 and $\frac{|U_\fkm \mu|}{\left\langle \chi_\fkm^\mu, \chi_\fkm^\mu \right\rangle_{U_\fkm}} = \frac{|U_\fkm \mu U_\fkm|}{|U_\fkm \mu|}$.
\end{proof}

\def\proj{\mathrm{proj}}

Between any two nilpotent $\FF_q$-algebras $\fkm \subset \fkn$, we can insert a finite sequence of subalgebras 
\[ \fkm = \fkn_0 \subset \fkn_1 \subset \dots \fkn_k = \fkn\] such that $\fkn_{i-1}$ is an ideal in $\fkn_i$ of codimension one.  In particular, by Corollary 6.2 in \cite{DI06}, we can take each $\fkn_{i-1}$ to be a maximal proper subalgebra of $\fkn_{i}$.  The normal sequence of algebra groups
\[ U_\fkm = U_{\fkn_0} \vartriangleleft U_{\fkn_1}   \vartriangleleft \dots \vartriangleleft U_{\fkn_k} = U_{\fkn}\] is then supernormal.  Thus, specializing Theorem \ref{resind-thm} to the case when $\fkm \subset \fkn$ is a subalgebra of codimension one will tell us in some sense how to compute the restriction of a supercharacter to any algebra subgroup.

In this direction, we begin by recalling an analogous result for irreducible characters.  Suppose $G$ is a finite group and $p$ is the smallest prime dividing $|G|$. If $H$ is a subgroup of index $p$, then $H$ is automatically normal \cite[Lemma I.6.7]{Lang} and the irreducible characters $\chi$ of $G$ restrict to $H$ in one of two ways.  In particular, if we let $\gamma$ be a nontrivial irreducible character of the abelian quotient $G/H \cong \ZZ/p\ZZ$ and denote by $\wt \gamma$ its lift to $G$
(that is, $\wt \gamma(g) = \gamma(gH)$ for $g \in G$), then one of the following occurs:
\begin{enumerate}
\item[(1)] If $\chi \in \Irr(G)$ has $\chi \otimes \wt\gamma = \chi$, then $\chi \downarrow H$ decompose as a sum of $p$ irreducible characters.
\item[(2)] If $\chi \in \Irr(G)$ has $\chi\otimes \wt\gamma \neq \chi$, then $\chi \downarrow H$ is irreducible.
\end{enumerate}
To provide a supercharacter analogue for this result, we note (as above) that if $\fkn$ is a nilpotent $\FF_q$-algebra and $\fkm$ is a subalgebra of codimension one, then $\fkm$ is maximal, and hence a two-sided ideal by \cite[Corollary 6.2]{DI06}.  We now have the following proposition.

%


\begin{proposition}\label{codim1-prop}
  Let $\fkm \subset \fkn$ be two nilpotent  $\FF_q$-algebras, and suppose $\fkm$ has codimension one in $\fkn$.  Choose a nontrivial supercharacter $\gamma$ of the algebra group $U_{\fkn/\fkm}$ and let $\wt \gamma$ denote its lift to $G$ is the sense of Proposition \ref{lift-prop}.  If $\chi$ is a supercharacter of $U_\fkn$, then the following statements hold:
  
  \begin{enumerate}
 \item[(1)]   $\chi \downarrow U_\fkm$ is equal to $q^a$ times the sum of $q^b$ distinct supercharacters of $U_\fkm$ of the same degree, where $a,b$ are nonnegative integers with $a+b\leq 2$.  
 
 \item[(2)] If $\chi$ is irreducible, then $a = 0$ and $b \in \{0,1\}$.
 
 \item[(3)] 
$\chi \downarrow U_\fkm$ is a supercharacter of $U_\fkm$ (i.e., $a=b=0$) if $\chi \otimes \wt \gamma \neq \chi$ .   
\end{enumerate}
 
\end{proposition}

\begin{remark} In contrast to the irreducible case, the condition $\chi \otimes \wt \gamma \neq \chi$ in (3) is not necessary for $\chi \downarrow U_\fkm$ to be a supercharacter of $U_\fkm$.  For example, if $\chi$ is irreducible and $\chi \otimes \wt \gamma = \chi$, then by the result in the irreducible case $\chi \downarrow U_\fkm$ is a sum of $q$ irreducible characters.  It can occur that this sum is equal to a single supercharacter of $U_\fkm$.  For example, if 
\[ \fkm = \left\{ \(\barr{ccc } 0 & a & b \\ & 0 & a \\ & & 0 \earr\) : a,b \in \FF_q \right\}
 \qquad\text{and}\qquad
  \fkn = \fkn_3(q)=
  \left\{ \(\barr{ccc } 0 & a  & b \\ & 0 & c \\ & & 0 \earr\) : a,b,c \in \FF_q \right\}\]
then taking $\chi = \chi_\fkn^{\lambda}$ and $\wt \gamma = \chi_\fkn^{\alpha}$ where $\lambda = e_{13}^* \in \fkn^*$ and $\alpha = e_{12}^* - e_{23}^* \in \fkm^\perp$, we have using   Corollary 4.7 in \cite{T09}   that $\chi \otimes \wt \gamma = \chi$, but  by Lemma \ref{5.1} below $\chi \downarrow U_\fkm = \chi_\fkm^{\mu}$ is a supercharacter, where $\mu = \lambda \downarrow \fkm$.
\end{remark}
 
 We prove the proposition using the following lemma, which gives a simple method of determining whether $\chi$ restricts to a supercharacter once we have chosen representative maps in $\fkn^*$ for $\chi$ and $\wt \gamma$.

\begin{lemma}\label{5.1}
Retaining the notation of Proposition \ref{codim1-prop}, let $\chi = \chi_\fkn^\lambda$ for some  $\lambda \in \fkn^*$.  Let $\mu = \lambda\downarrow \fkm \in \fkm^*$, and choose a nonzero element $\alpha \in \fkm^\perp = \{ \eta \in \fkn^* : \ker \eta \supset \fkm \}$.  Define integers $\delta_{\mathrm L}, \delta_{\mathrm R}, \delta_{\mathrm L}',\delta_{\mathrm R}' \in \{0,1\}$ by  
\[ \barr{ccc} \delta_{\mathrm L} = \left\{\barr{ll} 1,&\text{if } \lambda +\alpha \in U_\fkn \lambda, \\ 0,&\text{otherwise}, \earr\right. & &
  \delta_{\mathrm L}' = \left\{\barr{ll} 1,&\text{if } \lambda +\alpha \in U_\fkm \lambda, \\ 0,&\text{otherwise}, \earr\right.
 \\ &
 \\
    \delta_{\mathrm R} = \left\{\barr{ll} 1,&\text{if } \lambda +\alpha \in \lambda U_\fkn , \\ 0,& \text{otherwise}, \earr\right. & &
 \delta_{\mathrm R}' = \left\{\barr{ll} 1,&\text{if } \lambda +\alpha \in \lambda U_\fkm . \\ 0,& \text{otherwise} . \earr\right.
\earr \]  We then have
\[q^{\delta_{\mathrm L}} =  \frac{|U_\fkn \lambda|}{|U_\fkn \mu|},\qquad q^{\delta_{\mathrm R}} =  \frac{|\lambda U_\fkn |}{| \mu U_\fkn|},\qquad  q^{\delta'_{\mathrm L}} = \frac{| \mu U_\fkn|}{| \mu U_\fkm|},\qquad q^{\delta'_{\mathrm R}} = \frac{|U_\fkn \mu|}{|U_\fkm \mu|},\] and $a+b = \delta_{\mathrm L} + \delta_{\mathrm R}' = \delta_{\mathrm L}' + \delta_{\mathrm R}$.  Consequently $a+b = 0$ if and only if $\lambda + \alpha \notin U_\fkn \lambda \cup \lambda U_\fkn$ and $a+b = 2$ if and only if $\lambda + \alpha \in U_\fkm \lambda \cap \lambda U_\fkm$.

  \end{lemma}

\begin{proof}
We first note that
$\fkm^\perp = \FF_q \spanning\{\alpha\}$, since the dimension of $\fkm^\perp$ is the codimension of $\fkm $ in $\fkn$, which equals one.  Let $f: \fkn^* \rightarrow \fkm^*$ denote the restriction map $f(\eta) = \eta \downarrow \fkm$; then $f$ is a linear surjection with kernel $\fkm^\perp$ of cardinality $q$.  Next let $V = (U_\fkn \lambda - \lambda) \subset \fkn^*$ and $W = (U_\fkn \mu - \mu) \subset \fkm^*$.  Then both sets are $\FF_q$-vector spaces (by the usual arguments; see \cite[Lemma 4.2]{DI06}) and $f(V) = W$, so $|W| =|V|/ |\ker (f\downarrow V)| $. Since $\lambda + \alpha \in U_\fkn \lambda$ if and only if $\alpha \in V$, and since $\ker(f\downarrow V)=\ker(f) = \fkm^\perp$ if and only if $\alpha \in V$, 
it follows that $  \frac{|U_\fkn \lambda|}{|U_\fkn \mu|} = \frac{|V|}{|W|} = |\ker(f\downarrow V)| = q^{\delta_{\mathrm L}}$.  The formula for $q^{\delta_{\mathrm R}}$ follows by the same argument switched from left to right.

Next, we claim that 
\be\label{claim} \frac{|U_\fkn \mu|}{ | U_\fkm \mu |} = \left\{\barr{ll} 1,& \text{if there exists $g \in U_\fkn - U_\fkm$ with $g\mu = \mu$,} \\  q, & \text{otherwise.}\earr\right.\ee  To see this note that if no such $g$ exists then for any choice of representatives $g_1,\dots,g_q \in U_\fkn$ of the (right) cosets of $U_\fkm$ in $U_\fkn$, the sets $U_\fkm g_i \mu$ are disjoint and of equal cardinality, meaning $|U_\fkn \mu| = q|U_\fkm \mu|$.  On the other hand, suppose there exists some $g \in U_\fkn - U_\fkm$ with $g\mu = \mu$, so that $g = (1+G)^{-1}$ for some $G\notin \fkm$.  The elements $g_t \overset{\mathrm{def}}= (1+tG)^{-1}$ for $t \in \FF_q$ then form a set of representatives of the distinct right cosets of $U_\fkm$ in $U_\fkn$.  This follows since the cosets $U_\fkm g_t^{-1}$ are disjoint, as for any $M \in \fkm$ we have $(1+M)g_t^{-1} = 1 + (tG + \wt M)$ where $\wt M = Mg_t^{-1} \in \fkm$.  Therefore the cosets $U_\fkm g_t$ are disjoint, and furthermore, $g_t\mu = \mu + t(g\mu-\mu)=\mu$ for each $t \in \FF_q$.  Consequently,  $U_\fkm \mu = U_\fkn \mu$ as $U_\fkn = \bigcup_{t \in \FF_q} U_\fkm g_t$.  Thus, in this second case, $\frac{|U_\fkn \mu|}{ | U_\fkm \mu |} = 1$, which proves (\ref{claim}).

Now, we claim there exists $g \in U_\fkn - U_\fkm$ with $g\mu = \mu$ if and only if $S \not\subset \fkm$, where $S$ is the subspace of $\fkn$ defined by 
\[\ba S &\overset{\mathrm{def}}= \{ X \in \fkn : (\lambda h - \lambda)(X) = 0\text{ for all }h \in U_\fkm \} 
= \{ X \in \fkn : \lambda(XM) = 0\text{ for all }M \in \fkm \}
\\& = \{ X \in \fkn : \mu(XM) = 0\text{ for all }M \in \fkm\}.
\ea\]  Here the last equality follows by noting that $XM \in \fkm$ for all $X \in \fkn$ and $M \in \fkm$, so by definition $\lambda(XM) = \mu(XM)$.  Our claim now follows by noting that $g = (1+X)^{-1} \in U_\fkn$ has $g\mu = \mu$ if and only if $X \in S$.  Using the fact that $S$ is a subspace, one can check that 
$S\not\subset \fkm$ if and only if $(\lambda U_\fkm - \lambda) \not\supset \fkm^\perp$, which is equivalent to the condition $\lambda +\alpha \notin \lambda U_\fkm $ since $\fkm^\perp$ is 1-dimensional.  We therefore conclude that  $\frac{|U_\fkn \mu|}{ | U_\fkm \mu |} = q^{\delta_{\mathrm R}'}$.  As before, the formula for $q^{\delta_{\mathrm L}'}$ follows by symmetric arguments.

We now have  $a+b = \delta_{\mathrm L} + \delta_{\mathrm R}' = \delta_{\mathrm L}' + \delta_{\mathrm R}$ since
\[q^{a+b} = \frac{|U_\fkn \lambda|}{|U_\fkm \mu|} = \frac{|U_\fkn \lambda|}{|U_\fkn \mu|} \frac{|U_\fkn \mu|}{|U_\fkm \mu|}= \frac{|\lambda U_\fkn |}{|\mu U_\fkm |} = \frac{| \lambda  U_\fkn |}{|\mu U_\fkn|} \frac{|\mu U_\fkn|}{|\mu U_\fkm|}\]  Thus $a+b = 0$ iff $\delta_{\mathrm L} = \delta_{\mathrm R} = \delta_{\mathrm L}' = \delta_{\mathrm R}' = 0$, which is equivalent to the condition $\lambda+\alpha \notin U_\fkn \lambda \cup \lambda U_\fkn$, and $a+b= 2$ iff $\delta_{\mathrm L} = \delta_{\mathrm R} = \delta_{\mathrm L}' = \delta_{\mathrm R}'= 1$, which is equivalent to the condition $\lambda+\alpha \in U_\fkm\lambda \cap \lambda U_\fkm$.
\end{proof}

We now prove the proposition.

\begin{proof}[Proof of Proposition \ref{codim1-prop}]
(1) follows immediately from the lemma, and (2) comes from our result in the irreducible case.  Explicitly, if $\chi$ is irreducible then it restricts to a sum of either 1 or $q$ irreducible characters; in the first case $\chi \downarrow U_\fkm$ is a supercharacter, and in the second $\chi\downarrow U_\fkm$ is the sum of either 1 or $q$ supercharacters.

To prove (3), choose $\alpha \in \fkn^*$ such that $\wt \gamma = \chi_\fkn^\alpha$.  
Note from the proof of Proposition \ref{lift-prop} that $\alpha \in \fkm^\perp$, and since $\gamma$ is nontrivial, $\alpha \neq 0$.  Also, observe that since $\fkm$ is an ideal of codimension one and $\fkn$ is nilpotent, $\fkn^2 \subset \fkm$.  Consequently $U_\fkn \alpha U_\fkn = \{ \alpha\}$, and by (\ref{formula}) it follows that $\chi\otimes\wt \gamma = \chi_\fkn^{\lambda+\alpha}$.  Now suppose $\chi \otimes \wt \gamma \neq \chi$, so that $\lambda+\alpha \notin U_\fkn \lambda U_\fkn$.  Then clearly $\lambda +\alpha \notin U_\fkn\lambda\cup \lambda U_\fkn$, so by the lemma $a+b = 0$ and $\chi \downarrow U_\fkm$ is a supercharacter of $U_\fkm$.
 \end{proof}
 
\def\sgn{\mathrm{sgn}}

\begin{example} \emph{Alternating Pattern Groups.}  Given a poset $\cP$ on $[n]$, let $\cP^\cov$ denote the subset of covers in $\cP$; i.e., elements $(i,k) \in \cP$ for which no $j$ exists with $(i,j),(j,k) \in \cP$.  Define a map $\sgn : U_\cP \rightarrow \FF_q^+$ by \[\sgn(g) = \sum_{(i,j) \in \cP^\cov} g_{ij},\qquad\text{for }g\in U_\cP.\] One can check that $\sgn$ is a homomorphism, and that $\theta \circ t\hspace{0.5mm} \sgn$ defines a 1-dimensional representation for all $t\in \FF_q$.  Define the \emph{alternating pattern subgroup} \[A_\cP = \left\{ g\in U_\cP : \sum_{(i,j) \in \cP^\cov} g_{ij} = 0\right\}\] as the kernel of $\sgn$.  The group $A_\cP\subset U_\cP$ is an algebra subgroup of codimension one, so we can apply the preceding proposition and lemma.   Any $\alpha$ in the sense of Lemma \ref{5.1} is a multiple of $\alpha = \sum_{(i,j) \in \cP^\cov} e_{ij}^* \in \fkn_\cP^*$, and $\lambda +\alpha \notin U_\cP \lambda \cup \lambda U_\cP$ for all $\lambda \in \fkn_\cP^*$.  To see this, let $(i,j) \in \cP^\cov$ with $i$ minimal.  Then $ge_{ij} = e_{ij}$ for all $g \in U_\cP$ so 
\[ (g \lambda)_{ij} = \lambda_{ij} \neq \lambda_{ij} +1 = (\lambda+\alpha)_{ij},\qquad\text{for all }g\in U_\cP.\]  Therefore $\lambda+\alpha \notin U_\cP \lambda$, and a similar argument using $(i,j) \in \cP^\cov$ with $j$ maximal shows that $\lambda+\alpha \notin \lambda U_\cP$.

Thus, in analogy with the alternating subgroup of the symmetric group, every supercharacter of $U_\cP$ restricts to a supercharacter of $A_\cP$ by Lemma \ref{5.1}, and every supercharacter of $A_\cP$ arises in this way.  In addition, it follows from (\ref{formula}) that two supercharacters $\chi, \psi$ of $U_\cP$ have the same restriction  to $A_\cP$ if and only if $\chi = \psi \otimes \(\theta \circ t\hspace{0.5mm}\sgn\)$ for some $t \in \FF_q$.  More descriptively, we recall that two supercharacters are equal if and only if they are indexed by linear functionals in the same two-sided orbit.  If we let $\fka_\cP = \{ X \in \fkn_\cP : 1+X \in A_\cP\}$, then $\chi^\lambda \downarrow A_\cP = \chi^{\lambda \downarrow \fka _\cP}$ for $\lambda \in \fkn_\cP^*$.  Thus $\chi,\psi$ have the same restriction if and only if they can be  indexed by functionals in $\fkn_\cP^*$ which differ by a multiple of $\alpha$, since $\FF_q\spanning\{\alpha\}$ is the kernel of the restriction map $\fkn_\cP^* \rightarrow \fka_\cP^*$.  Since $\theta \circ t\hspace{0.5mm}\sgn = \chi^{t\alpha}$, and since $\chi^\lambda \otimes \chi^{t\alpha} = \chi^{\lambda + t\alpha}$ (which follows from (\ref{formula}) and the fact that $U_\cP \alpha U_\cP = \{\alpha\}$), our claim follows.

Let $A_n(q)$ denote the alternating pattern subgroup of $U_n(q)$.  Using the preceding observations, we can produce a formula for the number of supercharacters of $A_n(q)$.  We first require some definitions and a lemma.  Given a positive integer $n$, let
\[ \ba \mathscr{F}_n(q) &=  \{ \lambda \in \sP_n(q) : \lambda \text{ has a nonzero entry in the $i$th row or $i$th column for all $1\leq i \leq n$}\}, \\
F_n(q) & = |\mathscr{F}_n(q)|.\ea\] Here by convention $\mathscr{F}_0(q) = \{\varnothing\}$ and $F_0(q) = 1$.  Recall from Section \ref{sc-U_n-sect} the correspondence between elements of $\sP_n(q)$ and $\FF_q$-labeled set partitions of $[n]$.   The set $\mathscr{F}_n(q)$ corresponds to the subset of \emph{feasible} $\FF_q$-labeled set partitions of $[n]$, which are set partitions with no parts containing just one element.  The numbers $F_n(2)$ define the sequence \htmladdnormallink{A000296} {http://oeis.org/A000296} in \cite{OEIS}.  A survey of the combinatorial interpretations of $F_n(2)$ appear in \cite{Be99}, where $F_n(2)$ is the sequence $V_n$.  The following lemma gives a formula for $F_n(q)$ involving the Bell numbers.

\begin{lemma} \label{fpdef} The number $F_n(q)$ is the binomial transform of $B_n(q)$; i.e.,
  \[  F_n(q) = \displaystyle\sum_{k=0}^n (-1)^{k} \binom{n}{k}  B_{n-k}(q),\qquad\text{for }n\geq 0.\]
\end{lemma}

\begin{proof} The proof is the same as that of Lemma \ref{count2}.  The statement holds if $n=0$; assume $n>0$ so that
$F_n(q) = B_n(q) - |\cT|$, where $\cT$ denotes the set  of elements in $\sP_n(q)$ which have all zeros in the $i$th row and column for at least one $i$.  Given a nonempty subset $S \subseteq \{1,\dots,n\}$, let $\cT_{S} \subseteq \cT$ denote the set of $\lambda \in \sP_n(q)$ with $\lambda_{ix} = \lambda_{xi} = 0$ for all $i \in S$ and $1\leq x \leq n$.  For any such $S$ we have $|\cT_{S}| = B_{n-|S|}(q)$, since deleting the rows and columns with coordinates in $S$ gives a bijection $\cT_S \rightarrow \sP_{n-|S|}(q)$.   Furthermore,  if $R,S\subseteq [[n]]$ are two nonempty subsets, then one sees directly that $\cT_R \cap \cT_S = \cT_{R\cup S}$.  Since $\cT = \bigcup_{i=1}^n \cT_{\{i\}}$, the desired result follows from the inclusion-exclusion principle.
\end{proof}

We now have an explicit formula.

\begin{proposition} The number of supercharacters of $A_{n+1}(q)$ is
\[ \frac{1}{q}B_{n+1}(q) + \frac{q-1}{q}F_n(q) =  \sum_{k=0}^n \frac{(q-1)^k + (-1)^k(q-1)}{q} \binom{n}{k} B_{n-k}(q),\qquad\text{for $n\geq 0$.}\]
\end{proposition}

\def\sF{\mathscr{F}}

\begin{proof}
Retaining the notation above, we have $\chi^\lambda \otimes \theta \circ t\hspace{0.5mm}\sgn = \chi^\lambda$ if and only if $\lambda +\alpha \in U_n(q) \lambda U_n(q)$, where $\alpha = \sum_{i=1}^{n-1} e_{i,i+1}^* \in \fkn_n^*(q)$.  One can check that this latter condition holds for $\lambda \in \sP^*_n(q)$ if and only if $\lambda$ has a nonzero entry strictly to the right of or strictly above $(i,i+1)$ for all $i$; call the set of such functionals $\wt\sF_{n}(q)$.  The cardinality of $\wt\sF_n(q)$ is then $F_{n-1}(q)$, since if we identify elements of $\wt \sF_n(q)$ as matrices in $\fkn_n(q)$, then deleting the first column and last row defines a bijection between $\wt\sF_n(q)$ and $\sF_{n-1}(q)$.  Each supercharacter $\chi^\lambda$ for $\lambda \in \wt\sF_n(q)$ restricts to a distinct supercharacter of $A_n(q)$.  Conversely, if $\lambda \in \sP^*_n(q) - \wt\sF_n(q)$ then the $q$ supercharacters $\chi^\lambda \otimes \theta \circ t\hspace{0.5mm}\sgn$ for $t\in \FF_q$ all have the same restriction.  Thus the number of supercharacters of $A_n(q)$ is \[\frac{1}{q} |\sP^*_n(q) - \wt\sF_n(q)| + |\wt\sF_n(q)| = \frac{1}{q} |\sP^*_n(q)| + \frac{q-1}{q} |\wt\sF_n(q)| = \frac{1}{q}B_n(q) + \frac{q-1}{q} F_{n-1}(q)\] and the second formula follows from the preceding lemma.
\end{proof}

There is a natural indexing set for the supercharacters of $A_n(q)$ given by  all $\FF_q$-labeled set partitions $\lambda$ of $[n]$ satisfying the following condition: if  the numbers $j$ and $j+1$ belong to the same part of $\lambda$, then for some $i<j$, $i$ is the largest element of its part in $\lambda$ and $i+1$ is the smallest element of its part in $\lambda$.  These set partitions correspond to the subset
\[ \bigl\{ \lambda \in \sP_n^*(q) : \lambda_{j,j+1} \neq 0 \text{ implies } \exists i\text{ with }1\leq i<j\text{ such that }\lambda_{i,x} = \lambda_{x,i+1}=0\text{ for all }x\bigr\}.\]
This follows by choosing an appropriate set of representatives of the equivalence classes in $\sP_n^*(q)$ under the relation $\sim$ defined by 
\[ \barr{c} \lambda \sim \mu\qquad \text{if and only if}\qquad  \lambda + t\(\sum_{i=1}^{n-1} e_{i,i+1}^*\) \in U_n(q) \mu U_n(q)\text{ for some $t \in \FF_q$}.
\earr\]   As noted in the discussion above, these equivalence classes parametrize the distinct restrictions of supercharacters of $U_n(q)$ to $A_n(q)$.


\end{example}

\section{Supercharacters of Abelian Semidirect Products}\label{abelian-sect}

Suppose $G$ is a finite group given by a semidirect product of the form $G = H\ltimes A$ where $A$ is normal and abelian.  Then $H$ acts on the set of irreducible character of $A$ by conjugation:
\[ h\cdot \tau(a) = \tau(h^{-1}ah),\qquad\text{for }h\in H,\ a \in A,\ \tau \in \Irr(A).\]  Mackey's ``method of little groups'' bijectively assigns to each irreducible character of $G$ a pair consisting of an $H$-orbit of $\Irr(A)$ and an irreducible character of the corresponding stabilizer subgroup in $H$.  The goal of this section is to provide a supercharacter analogue of this result for algebra groups given by semidirect products with an abelian supernormal subgroup.

In order to get some idea of what such an analogue might look like, let us describe how the irreducible characters of $G$ are parametrized more explicitly.  
Fix a set $\cR$ of representatives of the distinct $H$-orbits of irreducible characters of $A$, and for each $\tau \in \cR$ let $S_\tau$ denote its stabilizer subgroup in $H$.  Then each irreducible character of $G$ corresponds to a unique pair $(\cO_\tau, \psi)$ where $\cO_\tau$ is the $H$-orbit of some $\tau \in \cR$ and $\psi \in \Irr(S_\tau)$.  In particular, we have a bijection
\be \label{clifford-orig} \barr{ccc}
\Irr(G) & \to & \biggl \{ ( \cO_\tau,  \psi) : 
\text{$\tau \in \cR$ and $\psi \in \Irr(S_\tau)$}
\biggr\} \\
\chi & \mapsto & (\cO_\tau,\psi)
\earr \ee where $\chi$ is given by the explicit formula 
\[\chi = \Ind_{AS_\tau}^G \(\tilde \psi \otimes \tilde \tau\). \] Here $\tilde \psi$ and $\tilde \tau$ are the characters of $A S_\tau$ defined by $\tilde \psi(as) = \psi(s)$ and  $\tilde \tau(as) = \tau(a)$ for $a \in A$, $s \in S_\tau$.  Since $|\cO_\tau| = \frac{|G|}{|AS_\tau|}$, we can write this formula equivalently as
\be\label{formula-orig} m_\chi \chi = \Ind_{AS_\tau}^G \( m_\psi |\cO_\tau| \tilde \psi \otimes \tilde \tau\),\quad\text{where }m_\psi = \psi(1) = \frac{\psi(1)}{\langle \psi,\psi\rangle_{G}},\ m_\chi = \chi(1) = \frac{\chi(1)}{\langle \chi,\chi\rangle_{G}}.\ee  This version more closely mirrors its supercharacter analogue (\ref{clifford-bijection}) described below.

Theorem \ref{main-thm} describes a similar bijection for the supercharacters of an algebra group of the form $U_\fkn = U_\fkh \ltimes U_\fka$ where $U_\fka$ is supernormal and $\fka^2=0$.  In this case we again have a natural action of the subgroup $U_\fkh$ on the supercharacters of the abelian algebra group $U_\fka$, but this time this action is two-sided instead of by conjugation.  As before, the supercharacters of $U_\fkn$ are parametrized by the resulting $U_\fkh$-orbits and some additional data related to the characters of the corresponding stabilizer subgroups.  Unlike the irreducible case, however, this additional data takes the form of an equivalence class of supercharacters rather than a single supercharacter.

For the duration of this section, let $\fkn, \fkh, \fka$ be nilpotent $\FF_q$-algebras such that $U_\fkn = U_{\fk h} \ltimes U_{\fk a}$ is a semidirect product of algebra groups with $U_\fka$ supernormal and $\fka^2=0$.  Observe that in this case $\fkn = \fk h \oplus \fk a$ as a vector space, $U_\fka$ is abelian, and $\fk a$ is a two-sided ideal.  If $U_\fka$ is a pattern group then $U_\fka$ is abelian if and only if $\fka^2=0$, but this does not hold for algebra groups in general.
Given any subspace $\fk m \subset \fk n$, let \[\fk m^\perp = \{ \gamma \in \fkn^* : \ker \gamma \supset \fk m\}.\]  
Then $\fkn^* = \fk h^\perp \oplus \fk a^\perp$ and we have natural vector space isomorphisms $\fkh^\perp \cong \fka^*$ and $\fka^\perp \cong \fkh^*$ given by restriction to $\fka$ and $\fkh$, respectively.

We lead up to our theorem classifying the supercharacters of $U_\fkn$ with two lemmas.  The first examines some of the special properties the structure of $U_\fkn$ imposes on the group's action on the dual space $\fkn^*$.  To state this result, we introduce the following notation.  Given $\alpha \in \fkh^\perp$, define 
\[ \fk l_\alpha = \{ H \in \fk h: H \fk a \subset \ker \alpha\},\qquad \fk r_\alpha = \{ H \in \fk h: \fk a H \subset \ker \alpha\},\qquad\text{and}\qquad \fk s_\alpha = \fk l_\alpha \cap \fk r_\alpha.\]  These sets are subalgebras of $\fkh$ as a consequence of the fact that $\fka$ is an ideal.  Let $L_\alpha = 1 + \fk l_\alpha$, $R_\alpha = 1 + \fk r_\alpha$, and $S_\alpha = 1 + \fk s_\alpha$ denote the corresponding algebra subgroups of $U_\fkh$.  In addition, let \[T_\alpha = \{ (g,h) \in U_{\fk h}\times U_{\fk h} : g\alpha h^{-1} = \alpha \}\] denote the stabilizer subgroup of $\alpha$ with respect to the two-sided action of $U_\fkh$.   We now have our first lemma.

\begin{lemma}\label{lemma6.3}  For any $\alpha \in \fkh^\perp$, the following hold:
\begin{enumerate}

\item[(1)] The groups $L_\alpha$ and $R_\alpha$ are the left and right stabilizers of $\alpha$ in $U_\fkh$, respectively.

\item[(2)] We have 
\[ \ba\(U_{\fk a} \alpha - \alpha\) &= \fk r_\alpha^\perp \cap \fk a^\perp, \\ 
\(\alpha U_{\fk a}- \alpha\) &= \fk l_\alpha^\perp\cap  \fk a^\perp,\ea\qquad\text{and}\qquad \(U_{\fk a} \alpha U_\fka   - \alpha\) = (U_{\fk a} \alpha - \alpha) + (\alpha U_{\fk a}- \alpha) = \fk s_\alpha^\perp \cap \fk a^\perp.\]  Consequently $|U_\fka \alpha U_\fka||S_\alpha| = |U_\fkh|$.
\item[(3)] 
For all $(g,h) \in T_\alpha$ and $X \in \fk s_\alpha$, we have $gX h^{-1} \in \fk s_\alpha$.  


\end{enumerate} 
\end{lemma}

\begin{proof}



Part (1) is quite similar to the first two parts of \cite[Lemma 4.2]{DI06}, and its proof follows largely the same argument.
For example, 
to see that $L_\alpha$ is the left stabilizer of $\alpha$ in $U_{\fk h}$, observe that if $H \in \fk h$, then $(1+H)\alpha = \alpha$ if and only if $(1+H)^{-1}\alpha =\alpha$ if and only if $\alpha(HX) = 0$ for all $X \in \fkn$.
Since $H\fk h \subset \fk h \subset \ker \alpha$ by definition and since $\fkn = \fk h + \fk a$, it follows that $(1+H)\alpha = \alpha$ if and only if $H \fk a \subset \ker \alpha$, in which case $1+H \in L_\alpha$. The proof that $R_\alpha$ is the right stabilizer of $\alpha$ in $U_{\fk h}$ is identical.  

Define subalgebras $\fk l_\alpha' = \{ A \in \fk a : A \fk h \subset \ker \alpha\}$ and $\fk r_\alpha' = \{ A \in \fk a : \fk h A \subset \ker \alpha\}$, and let $L_\alpha'$ and $R_\alpha'$ be the corresponding algebra subgroups of $U_\fkn$.  By similar arguments, it follows that $L_\alpha'$ and $R_\alpha'$ are the left and right stabilizers  of $\alpha$ in $U_{\fk a}$.

To prove that $(U_{\fk a} \alpha - \alpha) = \fk r_\alpha^\perp \cap \fk a^\perp$, we first observe that $(U_{\fk a} \alpha - \alpha) \subset \fk r_\alpha^\perp \cap  \fk a^\perp$ since for $a \in U_{\fk a}$, $A \in \fka$, and $H \in \fk l_\alpha$, we have $(a^{-1}-1)A =0 \Rightarrow(a \alpha  - \alpha)(A)  =0$ and $(a^{-1}-1)H \in \ker \alpha  \Rightarrow (a \alpha  - \alpha)(H) =0$.
Thus \[|U_{\fk a}\alpha - \alpha| = |U_\alpha \alpha| \leq |\fk r_\alpha^\perp \cap \fk a^\perp| = |\fk h| / |\fk r_\alpha| = |U_{\fk h}| / |R_\alpha| = |\alpha U_{\fk h}|.\]  
On the other hand, $(\alpha U_{\fk h} - \alpha) \subset (\fk l_\alpha')^\perp \cap \fk h^\perp$ since for $h \in U_{\fk h}$, $H \in \fkh$, and $A \in \fk l_\alpha'$, we have $H(h^{-1}-1) \in \fk h \subset \ker \alpha \Rightarrow (\alpha h - \alpha)(H) =0$ and $A(h^{-1}-1) \in \ker \alpha  \Rightarrow ( \alpha h  - \alpha)(A)  =0.$
Thus \[|\alpha U_{\fk h} - \alpha| = |\alpha U_{\fk h}| \leq |(\fk l'_\alpha)^\perp \cap \fk h^\perp| = |\fk a| / |\fk l'_\alpha| = |U_{\fk a}|/|L'_\alpha| = |U_{\fk a}\alpha|,\] so both of our inequalities become equalities throughout, and we obtain $|U_{\fk a}\alpha - \alpha| = |\fk r_\alpha^\perp \cap \fk a^\perp|$ and consequently $(U_{\fk a}\alpha - \alpha) = \fk r_\alpha^\perp \cap \fk a^\perp$.  The proof that $\alpha U_{\fk a}- \alpha = \fk l_\alpha^\perp\cap  \fk a^\perp$ is similar.  

It follows that 
\[(U_\fka \alpha - \alpha) + (\alpha U_\fka - \alpha) = \fk l_\alpha^\perp \cap \fka^\perp + \fk r_\alpha^\perp \cap \fka^\perp= (\fk l_\alpha^\perp + \fk r_\alpha^\perp) \cap \fka^\perp = (\fk l_\alpha \cap \fk r_\alpha )^\perp \cap \fka^\perp = \fk s_\alpha^\perp \cap \fka^\perp.\]
  But  observe that for all $a,b \in U_\fka$, $(a\alpha -\alpha) + (\alpha b - \alpha) = a\alpha b - \alpha$, since if $X \in \fkn$ then $(a^{-1}-1)X \in \fka \Rightarrow (a^{-1}-1)X (b^{-1}-1)   =0$, and so
\[ (a\alpha b-\alpha)(X) - (a\alpha -\alpha)(X) - (\alpha b - \alpha)(X) = \alpha((a^{-1}-1)X(b^{-1}-1)) =0.\]  Hence $(U_\fka \alpha U_\fka -\alpha) = (U_\fka \alpha - \alpha) + (\alpha U_\fka - \alpha) = \fk s_\alpha^\perp \cap \fka^\perp$.

Finally, suppose $(g,h) \in T_\alpha$ so that $g,h \in U_{\fk h}$ and $g\alpha h^{-1} = \alpha$.  Let $H \in \fk s_\alpha$, so that $H\fk a, \fk a H \subset \ker \alpha$.  Fix $A \in \fka$, and note that $g^{-1}Ag, h^{-1}Ah \in \fk a$  since $U_{\fk a}$ is supernormal.  Therefore $gHh^{-1} \in \fk s_\alpha$, since $\alpha(gH h^{-1}A) =  \alpha(H(h^{-1}Ah)) = 0$ and $\alpha(AgHh^{-1}) = \alpha((g^{-1} A g) H) =0$.  
 \end{proof}

Our next lemma uses the preceding results to say precisely when two functionals in $\fkn^*$ index the same supercharacter.  In order to state it, we observe that if $\alpha \in \fkh^\perp$, then it follows from (3) of the previous lemma that $T_\alpha$ acts on $\fk s_\alpha$ by $(g,h)\cdot X = gXh^{-1}$ and on its dual space $\fk s_\alpha^*$ by 
 \[ (g,h)\cdot \eta(X) = \eta(g^{-1} X h), \qquad\text{for }(g,h) \in T_\alpha,\ X \in \fk s_\alpha,\ \eta \in \fk s_\alpha^*.\] We now have the following.

\begin{lemma}\label{lemma6.4} 
Let $\alpha_1,\alpha_2,\alpha \in \fkh^\perp$ and $\eta_1,\eta_2,\eta \in \fka^\perp$.  Then the following hold:
\begin{enumerate} 
\item[(1)] $\chi_\fkn^{\alpha_1+\eta_1} = \chi_\fkn^{\alpha_2 +\eta_2}$ only if $\alpha_1 \in U_\fkh \alpha_2 U_\fk h$.

\item[(2)] $\chi_\fkn^{\alpha+\eta_1} = \chi_\fkn^{\alpha +\eta_2}$  if and only if $\eta_1\downarrow \fk s_\alpha \in T_\alpha\cdot(\eta_2\downarrow \fk s_\alpha)$. 

\item[(3)] If $\lambda = \alpha + \eta$, then $|U_\fkn \lambda U_\fkn| =|U_\fkh \alpha U_\fkh|  |U_\fka \alpha U_\fka| |T_\alpha\cdot (\eta\downarrow \fk s_\alpha)|$.
\end{enumerate}
\end{lemma}

\begin{proof}
Write $\lambda_i = \alpha_i + \eta_i \in \fkn^*$.  Since $U_{\fkn} = U_{\fk h}U_{\fk a} = U_{\fk a}U_{\fk h}$, we have $\chi_\fkn^{\lambda_1} = \chi_\fkn^{\lambda_2}$ if and only if 
$h_1a_1 \lambda_1 = \lambda_2 a_2h_2$ for some $a_i\in U_{\fk a},\ h_i \in U_{\fk h}$; by (2) of the previous lemma and Proposition \ref{supernormal-properties}, this is equivalent to 
\be \label{prev-eq} h_1a_1 \lambda_1 = \underbrace{h_1\alpha_1}_{\in \fkh^\perp} + \underbrace{h_1\((a_1\alpha_1-\alpha_1)+\eta_1\)}_{\in \fka^\perp} = \underbrace{\alpha_2h_2}_{\in \fkh^\perp} + \underbrace{\((\alpha_2a_2-\alpha_2)+\eta_2\)h_2}_{\in \fka^\perp} =\lambda_2a_2h_2.\ee  Since $\fkn^* = \fka^\perp \oplus \fkh^\perp$, this holds only if $h_1\alpha_1 =  \alpha_2 h_2$, in which case $\alpha_1 \in U_{\fk h} \alpha_2 U_{\fk h}$. This proves (1).

Now assume $\alpha_1 = \alpha_2 = \alpha$; then $\chi_\fkn^{\lambda_1}=\chi^{\lambda_2}_\fkn$ implies $h_1\alpha = \alpha h_2$ so $(h_1,h_2) \in T_\alpha$.  Using this fact, it follows, after acting on both sides of (\ref{prev-eq}) on the right with $h_2^{-1}$, that $\chi_\fkn^{\lambda_1} = \chi_\fkn^{\lambda_2}$ if and only if 
\[\alpha + (h_1a_1h_1^{-1} \alpha - \alpha) + h_1 \eta_1 h_2^{-1} = \alpha + (\alpha a_2 - \alpha) + \eta_2.\]  Since $U_\fka$ is normal, we can without loss of generality replace $h_1a_1h_1^{-1}$ with an arbitrary element of $U_\fka$.  Consequently, we have using (2) from the previous lemma that 
\[\ba \chi_\fkn^{\alpha+\eta_1} = \chi_\fkn^{\alpha+\eta_2}& \quad\Leftrightarrow\quad 
(a_1 \alpha - \alpha) + g \eta_1 h^{-1} = (\alpha a_2 - \alpha) + \eta_2\text{ for some }a_i\in U_{\fk a},\ (g,h) \in T_\alpha \\
& \quad\Leftrightarrow\quad  g \eta_1 h^{-1}  \in \eta_2 + \fk s_\alpha^\perp \cap \fka^\perp \text{ for some }(g,h) \in T_\alpha\\
& \quad\Leftrightarrow\quad  g (\eta_1\downarrow \fk s_\alpha) h^{-1}  \in \eta_2\downarrow \fk s_\alpha \text{ for some }(g,h) \in T_\alpha,
\ea\]  which proves (2).

Write $\lambda = \alpha + \eta$ with $\alpha \in \fkh^\perp$ and $\eta \in \fka^\perp$.  Suppose $U_\fkh \alpha U_\fkh$ has $N = |U_\fkh \alpha U_\fkh|$ distinct elements of the form $\alpha_i = g_i \alpha h_i^{-1}$ for $i=1,\dots,N$ where $g_i,h_i \in U_\fkh$.  Let $\eta_i = g_i \eta h_i^{-1}$; then it follows from (1) and (2) that 
\[ |U_\fkn \lambda U_\fkn|  =\sum_{i=1}^N |\fk s_{\alpha_i}^\perp \cap \fka^\perp ||T_{\alpha_i}\cdot(\eta_i\downarrow \fk s_{\alpha_i})|= \sum_{i=1}^N |U_\fka \alpha_i U_\fka||T_{\alpha_i}\cdot(\eta_i\downarrow \fk s_{\alpha_i})|.\]  Since $hU_\fka  = U_\fka h $ for all $h \in U_\fkh$, we have $|U_\fka \alpha_i U_\fka| = |g_i U_\fka \alpha U_\fka h_i^{-1}| =  |U_\fka \alpha U_\fka|$ for all $i$.  Since $T_{\alpha_i} = (g_i,h_i)T_{\alpha} (g_i,h_i)^{-1}$ and $\fk s_{\alpha_i} = g_i\fk s_\alpha h_i^{-1}$, it similarly follows that $|T_{\alpha_i} \cdot (\eta_i\downarrow \fk s_{\alpha_i})| = |T_\alpha \cdot (\eta\downarrow \fk s_\alpha)|$ for all $i$.  Hence $|U_\fkn \lambda U_\fkn| = \sum_{i=1}^N |U_\fka \alpha U_\fka| |T_\alpha \cdot (\eta\downarrow \fk s_\alpha)| = |U_\fkh \alpha U_\fkh| |U_\fka \alpha U_\fka| |T_\alpha \cdot (\eta\downarrow \fk s_\alpha)|$, proving (3).
\end{proof}

We can now describe a supercharacter analogue for Mackey's ``method of little groups.''    As above, we continue to let $\fkn, \fkh, \fka$ be nilpotent $\FF_q$-algebras such that $U_\fkn = U_{\fk h} \ltimes U_{\fk a}$ is a semidirect product of algebra groups with $U_\fka$ supernormal and $\fka^2=0$.  The set of supercharacters of the abelian algebra group $U_\fka$ coincides with the set of its irreducible characters, since $\fka^2=0$ implies that every supercharacter is linear.  The group $U_{\fk h}$ acts compatibly on the left and right on this set by the formula
\[g\tau h (1+A) = \tau(1+g^{-1}A h^{-1}),\qquad\text{for } g,h \in U_\fkh,\  A \in \fka,\ \tau \in \Irr(U_\fka).\]   Given a supercharacter $\tau$ of $U_{\fk a}$, let $L_\tau$ and $R_\tau$ denote its left and right stabilizers in $U_{\fk h}$ and set $S_\tau = L_\tau \cap R_\tau$.  In addition, let $T_\tau = \{ (g,h) \in U_{\fk h}\times U_{\fk h} : g\tau h^{-1} = \tau\}$ denote the stabilizer of $\tau$ in $U_{\fk h}\times U_{\fk h}$.  

If we write $\tau$ explicitly as $\tau = \chi_\fka^\alpha$ for some $\alpha\in \fka^*$, then  in the notation of the preceding lemmas, we have $L_\tau = L_\alpha$, $R_\tau = R_\alpha$, $S_\tau = S_\alpha$, and $T_\tau = T_\alpha$.  By Lemma \ref{lemma6.3}, we therefore can assert the following.   $S_\tau$ is an algebra group of the form $S_\tau = 1 + \fk s_\tau$ for a subalgebra $\fk s_\tau \subset \fk h$.  The algebra $\fk s_\tau$ is closed under the action of $T_\tau$, and so $T_\tau$ acts on its dual space $\fk s_\tau^*$ by the formula
 \[ (g,h)\cdot \eta(X) = \eta(g^{-1}Xh),\qquad\text{for }(g,h)\in T_\tau,\ X\in \fk s_\tau,\ \eta \in \fk s_\tau^*.\]  The orbits of this action consist of unions of two-sided $S_\tau$-orbits in $\fk s_\tau^*$ because $T_\tau$ contains $S_\tau\times S_\tau$ as a subgroup, and so we have an equivalence relation $\sim_\tau$ on the set of supercharacters of $S_\tau$ defined by 
\[ \chi_{\fk s_\tau}^\mu \sim_\tau \chi_{\fk s_\tau}^{\nu}\text{ if and only if }\mu \in T_\tau\cdot \nu\text{ for }\mu,\nu \in \fk s_\tau^*.\]


The following theorem now classifies the supercharacters of $U_\fkn$.

\begin{theorem}\label{main-thm}
 Let $\fkn, \fkh, \fka$ be nilpotent $\FF_q$-algebras such that $U_\fkn = U_{\fk h} \ltimes U_{\fk a}$ is a semidirect product of algebra groups with $U_\fka$ supernormal and $\fka^2=0$.  Fix a set $\cR$ of representatives of the distinct two-sided $U_{\fk h}$-orbits of supercharacters of $U_{\fk a}$.  
Each supercharacter of $U_\fkn$ then corresponds to a unique pair $(\cO_\tau,  \cC_\psi)$, where $\cO_\tau$ 
denotes the two-sided $U_\fkh$-orbit of a supercharacter $\tau \in \cR$, and $\cC_\psi$ 
denotes the $\sim_\tau$-equivalence class of a supercharacter $\psi$ of $S_\tau$.  In particular, the map
\be \label{clifford-bijection} \barr{ccc}
\biggl\{ \text{Supercharacters of $U_\fkn$}\biggr\} & \to & \biggl \{ (\cO_\tau,  \cC_\psi) : 
\text{$\tau \in \cR$ and $\psi$ a supercharacter of $S_\tau$}
\biggr\} 
\\ 
\chi_\fkn^{\lambda} & \mapsto & (\cO_\tau,  \cC_\psi ),\text{ where } \tau = \chi_{\fk a}^{\lambda \downarrow \fk a} \in \cR,\ \psi = \chi_{\fk s_\tau}^{\lambda \downarrow \fk s_\tau}
\earr \ee is a bijection, with inverse $(\cO_\tau, \cC_\psi) \mapsto \chi$ where $\chi$ is the supercharacter of $U_\fkn$ determined by the identity
\be\label{explicit-inverse}  m_\chi \chi = \SInd_{U_\fka S_\tau}^{U_\fkn}\bigg( \sum_{\vartheta \in \cC_\psi} m_\vartheta |\cO_\tau|\tilde \vartheta \otimes \tilde \tau \bigg),\quad\text{where } m_\vartheta = \frac{\vartheta(1)}{\langle \vartheta,\vartheta\rangle_{S_\tau}},\ m_\chi = \frac{\chi(1)}{\langle \chi,\chi\rangle_{U_\fkn}}.\ee Here $m_\chi$ and $m_\vartheta$ denote the multiplicities of $\chi$  and $\vartheta$ in the characters of $\CC U_\fkn$ and $\CC S_\tau$, and $\tilde \vartheta$ and $\tilde \tau$ are the characters of $U_\fka S_\tau$ defined by $\tilde \vartheta(as) = \vartheta(s)$ and  $\tilde \tau(as) = \tau(a)$ for $a \in U_\fka$, $s \in S_\tau$.

\end{theorem}

\begin{remark} We have defined the map (\ref{clifford-bijection}) only for supercharacters indexed by functionals $\lambda \in \fkn^*$ such that $\chi_{\fka}^{\lambda\downarrow \fka} \in \cR$.  
Given any $\lambda \in \fkn^*$, however, we can always find some $\gamma \in U_\fkn \lambda U_\fkn$ such that $\chi_\fka^{\gamma \downarrow \fka} \in \cR$.  This follows since each $\tau \in \cR$ is of the form $\chi_\fka^{\alpha }$, where $\alpha \in \fka^*$ ranges over a set of representatives of the distinct two-sided $U_{\fkh}$-orbits in $\fka^*$,    
one of which must lie in the orbit $U_\fkh (\lambda \downarrow \fka) U_\fkh = \{ \gamma \downarrow \fka : \gamma \in U_\fkh\lambda U_\fkh\}$.  
This observation ensures that the map (\ref{clifford-bijection}) as stated is in fact defined for all supercharacters of $U_\fkn$.
\end{remark}

\begin{proof}

Parts (1) and (2) of Lemma \ref{lemma6.4} show the map (\ref{clifford-bijection}) to be well defined.  
To prove that the map is a bijection, it suffices to derive the explicit inverse map $(\cO_\tau,\cC_\psi) \mapsto \chi$.  For this, write $\tau  = \chi_\fka^{\alpha\downarrow \fka}$ for $\alpha \in \fkh^\perp$ so that in our previous notation $S_\tau = S_\alpha$, $T_\tau = T_\alpha$, and $\fk s_\tau = \fk s_\alpha$.
Let $\psi = \chi_{\fk s_\tau}^{\eta_0}$ for some $\eta_0 \in \fk s_\tau^*$, and let $ \eta \in \fka^\perp$ be an arbitrary functional with $\eta \downarrow \fk s_\tau = \eta_0$.  Observe that $U_\fka S_\tau = U_{\fka + \fk s_\tau}$ since $U_{\fka+\fk s_\tau}$ contains both factor groups and has the same order as their product.  Finally, given $\gamma \in \fk s_\tau^*$, let $\tilde \gamma \in (\fka + \fk s_\tau)^*$ be the linear functional defined by $\tilde \gamma(A+H) = \gamma(H)$ for $A \in \fka$ and $H \in \fk s_\tau$.    

Now, fix an arbitrary element $e \in U_\fkn$.   By definition,
\[ \ba
\SInd_{U_\fka S_\tau}^{U_\fkn}\bigg( \sum_{\vartheta \in \cC_\psi} m_\vartheta |\cO_\tau| \tilde \vartheta \otimes \tilde \tau \bigg)(e) 
&=
\frac{|\cO_\tau|}{|U_\fkn||U_\fka||S_\tau|} \sum_{x,y \in U_\fkn} f(x(e-1)y)\ea\] where $f : \fkn \rightarrow \CC^\times$ is defined by  
\[ f(X) = \left\{\barr{ll} 
0,&\text{if }X \notin \fka + \fk s_\tau, \\ \\
\displaystyle \sum_{\gamma \in T_\tau\cdot \eta_0} \theta \circ(\tilde \gamma + \alpha)(X), & \text{if }X \in \fka + \fk s_\tau. 
\earr\right.
\] 
To simplify this formula we make two observations.  First, since $\fka+\fk s_\tau$ is closed under the action of $T_\tau$ and since $(\fka +\fk s_\tau)^\perp = \fka^\perp \cap \fk s_\tau^\perp= (U_\fka \alpha U_\fka - \alpha)$ by Lemma \ref{lemma6.3}, it follows by standard character orthogonality arguments that
\be\label{obs1} \frac{1}{|U_\fka |^2} \sum_{a,b \in U_\fka} \theta\circ (a\alpha b-\alpha)(gXh^{-1}) = \left\{\barr{ll} 0,&\text{if }X \notin \fka + \fk s_\tau, \\ 1,&\text{if } X\in \fka + \fk s_\tau.\earr\right.\qquad\text{for any }(g,h) \in T_\tau.\ee  
Second, since in addition $T_\tau$ fixes $\alpha$, it follows that if $X \in \fka + \fk s_\tau$ and $(g,h) \in T_\tau$ and $\gamma = (g,h)\cdot \eta_0$, then $(\tilde \gamma + \alpha)(X) = (\eta+\alpha)(g^{-1}Xh)$.
Thus 
\be\label{obs2} f(X) = \frac{|T_\tau\cdot \eta_0|}{|T_\tau|} \sum_{(g,h) \in T_\tau} \theta \circ(\eta + \alpha)(gXh^{-1}),\qquad\text{if }X \in \fka + \fk s_\tau. \ee
Multiplying (\ref{obs1}) and (\ref{obs2}) gives 
\[ f(X) = \frac{|T_\tau\cdot \eta_0|}{|U_\fka|^2|T_\tau|} 
\sum_{(g,h)\in T_\tau} \sum_{a,b \in U_\fka }\theta \circ (\eta +a\alpha b)(gXh^{-1}),\qquad\text{for \emph{all} }X\in \fkn.\] 
Substituting this into our initial formula, we obtain by removing redundant summations
\[\ba 
\SInd_{U_\fka S_\tau}^{U_\fkn}\bigg(\sum_{\vartheta \in \cC_\psi} m_\vartheta |\cO_\tau| \tilde \vartheta \otimes \tilde \tau \bigg)(e) 
&=
\frac{|\cO_\tau||T_\tau\cdot \eta_0|}{|U_\fkn||U_\fka|^3|S_\tau||T_\tau|} \sum_{x,y \in U_\fkn} \sum_{(g,h) \in T_\tau} \sum_{a,b \in U_\fka}\theta \circ (\eta + a\alpha b)(gx(e-1)yh^{-1})
\\&
=\frac{|\cO_\tau||T_\tau\cdot \eta_0|}{|U_\fkn||U_\fka|^3|S_\tau|}  \sum_{x,y \in U_\fkn} \sum_{a,b \in U_{\fk a}} \theta \circ (\eta + a\alpha b)(x(e-1)y).
\ea\] 
Write $\lambda =  \alpha + \eta \in \fkn^*$, and note that $a\lambda b=\eta + a\alpha b $ for $a,b \in U_\fka$ by (4) of Proposition \ref{supernormal-properties}.  Since $\frac{|U_\fka \alpha U_\fka|}{|U_\fkh|} = \frac{1}{|S_\tau|}$ by Lemma \ref{lemma6.3} and $|\cO_\tau||T_\tau\cdot \eta_0||U_\fka \alpha U_\fka| =
|U_\fkh \alpha U_\fkh||U_\fka \alpha U_\fka||T_\tau\cdot \eta_0|=
 |U_\fkn \lambda U_\fkn|$ by Lemma \ref{lemma6.4}, we have
\[\ba 
\SInd_{U_\fka S_\tau}^{U_\fkn}\bigg( \sum_{\vartheta \in \cC_\psi} m_\vartheta |\cO_\tau|\tilde \vartheta \otimes \tilde \tau \bigg)(e) 
&=
\frac{|\cO_\tau||T_\tau\cdot \eta_0|}{|U_\fkn||U_\fka|^3|S_\tau|} \sum_{x,y \in U_\fkn} \sum_{a,b \in U_{\fk a}} \theta \circ(x^{-1}a \lambda by^{-1})(e-1) 
\\&
=\frac{|\cO_\tau||T_\tau\cdot \eta_0||U_\fka \alpha U_\fka|}{|U_\fkn||U_\fka||U_\fkh|}  \sum_{x,y  \in U_\fkn} \theta \circ (x\lambda y)(e-1) 
\\& = \sum_{\mu \in U_\fkn \lambda U_\fkn} \theta\circ \mu(e-1) 
\\&= \frac{\chi_\fkn^\lambda(1)}{\langle \chi_\fkn^\lambda,\chi_\fkn^\lambda\rangle_{U\fkn}}\chi_\fkn^\lambda(e).
\ea\]   Since $\lambda \downarrow \fka = \alpha \downarrow \fka$ we have  $\chi_\fka^{\lambda \downarrow \fka} = \tau \in \cR$, and since $\lambda \downarrow \fk s_\tau = \eta\downarrow \fk s_\tau = \eta_0$ we have $\chi_{\fk s_\tau}^{\lambda\downarrow \fk s_\tau} = \psi$.  Hence $\chi_\fkn^\lambda \mapsto (\cO_\tau, \cC_\psi)$ under the map (\ref{clifford-bijection}), which completes the proof of the theorem.
 \end{proof}

To conclude this work, we show in an example how one can apply 
 Theorem \ref{main-thm} to describe the supercharacters of a particular algebra group.  In order to keep the technical considerations in this example to a minimum, we first state the following general purpose lemma.

Fix a positive integer $n$.  For any subset $J \subset [[n]]$  of positions above the diagonal in an $n\times n$ matrix, define a vector space $\fkn_J =\FF_q\spanning\{e_{ij} \in \fkn_n(q) : (i,j) \in J\}$ over $\FF_q$.  Suppose $L,R\subset  U_n(q)$ are algebra groups over $\FF_q$ such that 
\be \label{cond-curr} g X, X h \in \fkn_J,\qquad\text{for all $g \in L$, $h \in R$, $X \in \fkn_J$.}\ee  Then the product group $L\times R$ acts linearly on $\fkn_J$ by $(g,h)\cdot X = gXh^{-1}$ and in turn on its dual space $\fkn_J^*$ by $(g,h)\cdot\lambda(X) = \lambda(g^{-1}Xh)$ for $(g,h) \in L\times R$, $X \in \fkn_J$, and $\lambda \in \fkn_J^*$.  Given these observations, we have the following result.

\begin{lemma}\label{example-lemma}  Let $J\subset [[n]]$ and suppose $L,R\subset  U_n(q)$ are algebra groups over $\FF_q$ satisfying (\ref{cond-curr}).  
Assume for all $1\leq i <j<k<l \leq n$ the following condition holds:
\be \label{cond-6.3} (i,k),(j,k)\in J \text{ implies } 1 + e_{ij} \in L\qquad\text{and}\qquad 
(j,k),(j,l) \in J \text{ implies } 1 + e_{kl} \in R.\ee
Then the $L\times R$-orbits in $\fkn_J$ and in $\fkn_J^*$ are both in bijection with the set 
$\sP_n(q) \cap \fkn_J$.
\end{lemma}

\begin{proof}
By Lemma 4.1 in \cite{DI06} the number of $L\times R$ orbits in $\fkn_J$ is the same as the number of orbits in $\fkn_J^*$, so we need only to count the orbits in $\fkn_J$.  We know from the description (\ref{U_n(q)-classification}) of the superclasses of $U_n(q)$ that the elements $\lambda \in \sP_n(q) \cap \fkn_J$ all belong to distinct two-sided $U_n(q)$-orbits, and hence to distinct $L\times R$-orbits.  To show that these elements represent all the orbits in $\fkn_J$, we use a straightforward argument by induction.    

To set this up, we define a function $f: [[n]] \rightarrow \NN$ by 
\[\ba  f(i,j) &= i + \biggl(\ \sum_{t=1}^{n-1- (j-i)} t\ \biggr)
= i +\frac{(n-1-j+i)(n-j+i)}{2},\ea
\qquad\text{for }1\leq i<j\leq n.\]  This function just orders the positions above the diagonal in an $n\times n$ matrix; for example, if $n=4$ then 
\[ (f(i,j))_{1\leq i < j\leq n} = \(\barr{cccc} 
  & 4 & 2 & 1 \\
  &   & 5 & 3 \\
    &  &   & 6 \\
  &   &   &   \earr\).
\]  Now define $d(X) = 0$ if $X=0$ and $d(X) = \max\{ f(i,j): X_{ij} \neq 0\}$ otherwise.

We want to show that every $X \in \fkn_J$ belongs to the $L\times R$-orbit of some $\lambda \in \sP_n(q) \cap \fkn_J$.  
To do this, we induct on $d(X)$.  If $d(X) = 0$ then $X = 0$ and this is obvious, so assume $X \neq 0$ and that any $Y \in \fkn_J$ with $d(Y) < d(X)$ belongs to the orbit of some $\lambda \in \sP_n(q) \cap \fkn_J$.  Let $(j,k) \in \supp(X)$ be the position with $f(j,k) = d(X)$ and define 
\[ x = \prod_{i < j} \( 1 - X_{ik} (X_{jk})^{-1} e_{ij}\)\qquad\text{and}\qquad y = \prod_{\ell>k} \( 1 - X_{j\ell} (X_{jk})^{-1} e_{k\ell}\)
\] where the products (of commuting factors) are taken in any order.  For each $i<j$, the corresponding factor in $x$ lies in $L$ if $(i,k) \in J$ by (\ref{cond-6.3}), and is $1 \in L$ if $(i,k) \notin J$, so $x \in L$.  By similar reasoning, $y \in R$.  Therefore $xXy \in \fkn_J$ and, more significantly, one can check that $xXy = X_{jk}e_{jk}  + Y$ where $Y \in \fkn_J$ has all zeros in $j$th row and $k$th column, and has $d(Y) < d(X)$.

Consequently, by inductive hypothesis there exists $(g,h) \in L\times R$ with $gYh = \lambda  \in \sP_n(q) \cap \fkn_J$.  We may assume that $\lambda$ has no nonzero entries in the $j$th row or $k$th  column, since setting $g_{j\ell} = 0$ for all $\ell>j$ and $h_{ik} = 0$ for all $i<k$ has the effect of replacing row $j$ and column $k$ in $\lambda$ with zeros, in which case $gYh = \lambda$ remains an element of $\sP_n(q) \cap \fkn_J$.  Likewise, we may assume that $g_{ij} = 0$ for all $i<j$ and $h_{k\ell} = 0$ for all $\ell>k$ as these entries have no effect on the product $gYh$.  It follows from these assumptions that $ge_{jk}h = e_{jk}$ and in turn that $gxXyh = X_{jk}e_{jk} + \lambda  \in \sP_n(q) \cap \fkn_J$.  This proves by induction that the elements in $\sP_n(q) \cap \fkn_J$ index the distinct $L\times R$ orbits in $\fkn_J$.
\end{proof}

Before proceeding to our example, we introduce a final bit of notation.  
Given $1\leq i \leq n$ let 
\be\label{N_q-def}
\ba N_{n,i}(q) 
&=  \left|\left\{ \lambda\in \sP_n(q) : \lambda_{ij} = 0\text{ for all $j$}\right\}\right| = \left|\left\{ \lambda\in \sP_n(q) : \lambda_{j,n+1-i} = 0\text{ for all $j$}\right\} 
\right|.
 \ea\ee
The second equality follows by noting that the antitranspose map on $n\times n$ matrices defines  an involution of $\sP_n(q)$.  
Also, since $(q-1)N_{n,i}(q)$ is simply the number of $\lambda \in \sP_{n+1}(q)$ with a nonzero entry in position $(1,i)$, it follows that 
\be \label{N_q-id} B_n(q) + \sum_{i=1}^n (q-1)N_{n,i}(q) = B_{n+1}(q).\ee

\begin{example} Fix two positive integers $m$ and $n$.  Let $\H\subset \cP$ denote the posets on $[m+n]$ 
corresponding to the partial orderings
\be \label{posets}\H\ =\  \xy<0.45cm,1.5cm> \xymatrix@R=.3cm@C=.3cm{
  &      m+n\\
m &     \vdots \ar @{-} [u] &  \\
  \vdots \ar @{-} [u] &    m+1 \ar @{-} [u] \\
1 \ar @{-} [u]    &   \\
}\endxy
\qquad\text{and}\qquad
\cP\ =\  \xy<0.45cm,1.5cm> \xymatrix@R=.3cm@C=.3cm{
  &      m+n\\
m  \ar @{-} [ur] &     \vdots \ar @{-} [u] &  \\
  \vdots \ar @{-} [u] &    m+1 \ar @{-} [u] \\
1 \ar @{-} [u]   \ar @{-} [ur]   &   \\
}\endxy
 \ee
and set $\cA = \cP - \H = \{ (1, m+i), (j,m+n) : 1\leq i \leq  n,\ 1\leq j < m \}$. Then $\cA$ is also a poset on $[m+n]$ and $\cA \vartriangleleft \cP$.  Furthermore, since the partial ordering represented by $\cA$ contains no 3-chains, the algebra $\fkn_\cA$ satisfies $(\fkn_\cA)^2=0$.  Thus the pattern group $U_\cP = U_\H\ltimes U_\cA$ is given by a semidirect product of algebra groups with $U_\cA$ supernormal and $(\fkn_\cA)^2=0$, so Theorem \ref{main-thm} applies.  

Elements of the algebra $\fkn_\cP$ are $(m+n)\times(m+n)$ matrices of the form
\[X(a,b,c,x,y) \overset{\mathrm{def}}= \(\barr{r|c} x &  \barr{ll} a^T & c \\  0 & b \earr \\ \hline  0 &  y \earr\),\quad\text{where }x \in \fkn_m(q),\ y \in \fkn_n(q),\ a \in \FF_q^{n-1},\ b\in \FF_q^{m-1},\ c \in\FF_q.
\]  The subalgebra $\fkn_\H$ consists of all such matrices with $a=b=c = 0$, and so we have a natural isomorphism $U_\H \cong U_m(q) \times  U_n(q)$.  Likewise, the subalgebra $\fkn_\cA$ consists of all such matrices with $x=y = 0$.  We can naturally identify the dual space $\fkn_\cA^*$ with $\fkn_\cA$; under this identification, let $\alpha = \alpha(a,b,c) \in \fkn_\cA^*$ denote the linear functional corresponding to $X(a,b,c,0,0)$.  Each such $\alpha$ indexes a distinct supercharacter $\tau = \tau(a,b,c)$ of $U_\cA$ given by $\tau(g) = \theta \circ \alpha(g-1)$, and the action of $U_{\H}$ on the supercharacters of $U_\cA$ is equivalent to the group's action on $\fkn_\cA^*$.  

If we identify $\alpha$ with its corresponding matrix, then the right action of $U_\H$ on $\alpha$ adds multiples of entries in the first row to entries which are further to the left; i.e., we can add a multiple of $c$ to any $a_i$, or add a multiple of $a_j$ to each $a_i$ with $i<j$.  Similarly, the left action of $U_\H$ on $\alpha$ adds multiples of entries in the last column to entries which are further down.  It follows that the two-sided $U_\H$-orbits of the characters of $U_\cA$ are indexed by the set of $\alpha = \alpha(a,b,c)$ such that either $a=b=0$ and $c \in\FF_q^\times$, or $c=0$ and $a,b$ have at most one nonzero coordinate.  Thus, every $U_\H$ character orbit is indexed by a unique $\tau = \tau(a,b,c)$, where $a,b,c$ are described by one of the following five cases:

\begin{enumerate}
\item[(i)] $a=b=c =0$.   Then $T_\tau = S_\tau = U_\H$, and the $\sim_\tau$ equivalence classes $C_\psi$ of supercharacters of $S_\tau$ correspond to the distinct supercharacters of $U_\H$.  There are $B_m(q)B_n(q)$ of these, since $U_\H \cong U_{m}(q) \times  U_n(q)$.

\item[(ii)] $b=c = 0$ and $a = [0\ \cdots \ 0\ a_i \ 0 \cdots \ 0]^T \in \FF_q^{n-1}$ where $a_i \in \FF_q^\times$ and $1\leq i \leq n-1$.    Then $L_\tau = U_\H$ and $R_\tau$ is  the subgroup of $U_\H$ of matrices with no nonzero entries above the diagonal in the $(m+i)$th column, and it follows that $S_\tau = R_\tau$ and $T_\tau = U_\H\times R_\tau$.  After applying Lemma \ref{example-lemma} with $L = L_\tau$, $R = R_\tau$, and $\fkn_J = \fk s_\tau$, it follows that the $\sim_\tau$ equivalences classes  are indexed by pairs $(\lambda,\mu) \in \sP_m(q) \times \sP_n(q)$, where $\lambda$ is arbitrary, but we require that $\mu_{ki} = 0$ for all $k$.  There are $B_m(q) N_{n,n+1-i}(q)$ of these.

\item[(iii)] $a=c = 0$ and $b = [0\ \cdots \ 0\ b_j \ 0 \cdots \ 0]^T \in \FF_q^{m-1}$ where $b_j \in \FF_q^\times$ and $1\leq j \leq m-1$.    Then $R_\tau = U_\H$ and $L_\tau$ is  the subgroup of $U_\H$ of matrices with no nonzero entries above the diagonal in the $j$th row, and it follows that $S_\tau = L_\tau$ and $T_\tau = L_\tau\times U_\H$.  After applying Lemma \ref{example-lemma} as in (ii), it follows that the $\sim_\tau$ equivalences classes are again indexed by pairs $(\lambda,\mu) \in \sP_m(q)\times \sP_n(q)$, where this time $\mu$ is arbitrary, but we require that $\lambda_{j+1,k}=0$ for all $k$.  There are $ N_{m,j+1}(q)B_n(q)$ of these.

\item[(iv)] $c = 0$ and $a$, $b$ are given  as in Cases 2 and 3, respectively.  Then $R_\tau$ is given as in Case 2 while $L_\tau$ is given as in Case 3, and $T_\tau = L_\tau \times R_\tau$.  Thus $S_\tau$ consists of the subgroup of $U_\H$ of matrices with no nonzero entries above the diagonal in the $(m+i)$th column or $j$th row.  After applying Lemma \ref{example-lemma} with $L= L_\tau$, $R=R_\tau$, and $\fkn_J = \fk s_\tau$, it follows as before that the $\sim_\tau$ equivalences classes are indexed by pairs $(\lambda,\mu) \in \sP_m(q) \times \sP_n(q)$, where we require that $\lambda_{j+1,k} = \mu_{ki} = 0$ for all $k$.  There are $ N_{m,j+1}(q)N_{n,n+1-i}(q)$ of these.

\item[(v)] $a=b = 0$ and $c \in \FF_q^\times$.    Then $L_\tau \cong U_{m-1}(q)\times  U_n(q)$ and $R_\tau \cong U_{m}(q) \times U_{n-1}(q)$ are  the subgroups of $U_\H$ of matrices with no nonzero entries above the diagonal in the first row and last column, respectively, and $T_\tau = L_\tau \times R_\tau$.  Hence $S_\tau \cong U_{m-1}(q) \times U_{n-1}(q)$ and the $\sim_\tau$ equivalence classes correspond to the distinct supercharacters of $S_\tau$.  There are $B_{m-1}(q)B_{n-1}(q)$ of these.

\end{enumerate}

The set of supercharacters $\tau = \tau(a,b,c)$ corresponding to these five cases uniquely index the distinct $U_\H$-orbits of supercharacters of $U_\cA$.  Case (i) describes only one such $\tau$; cases (ii), (iii), and (v) each describe $q-1$; and case (iv) describes $(q-1)^2$.  Thus, it follows by counting the number of $\sim_\tau$ equivalence classes in each case and applying Theorem \ref{main-thm} that the number of supercharacters and superclasses of $U_\cP$ is 
\[ \( B_m(q) + \sum_{i=2}^{m} (q-1) N_{m,i}(q) \)\( B_n(q) + \sum_{j=2}^{n} (q-1) N_{n,j}(q)\) + (q-1)B_{m-1}(q)B_{n-1}(q),\] where $B_n(q)$ and $N_{n,i}(q)$ are defined by (\ref{B_q-def}) and (\ref{N_q-def}).    Since $N_{m,1}(q) = B_{m-1}(q)$, it follows from (\ref{N_q-id}) that we can express the number of supercharacters and superclasses just in terms of $q$-Bell numbers as follows:

\begin{proposition} The number of supercharacters and superclasses of the pattern group $U_\cP$, with $\cP$ defined as in (\ref{posets}), is
\be \biggl(B_{m+1}(q)-(q-1)B_{m-1}(q)\biggr)\biggl(B_{n+1}(q)-(q-1)B_{n-1}(q)\biggr) + (q-1) B_{m-1}(q)B_{n-1}(q).\ee
\end{proposition}
\end{example}

%
















\end{document}